\documentclass[14pt, draft]{article}
\usepackage[cp1251]{inputenc}
\usepackage[ukrainian, russian, english]{babel}

\usepackage{ amsfonts, amssymb}
\usepackage{amsmath}
\usepackage{amsthm}

\newtheorem{theorem}{Теорема}[section]

\newtheorem{lemma}[theorem]{Лема}
\newtheorem{corollary}[theorem]{Наслідок}

\theoremstyle{definition}

\begin{document}

\selectlanguage{ukrainian} \thispagestyle{empty}
 \pagestyle{myheadings}              

\pagestyle{myheadings}              

УДК 517.5 \vskip 3mm

\noindent \bf А.С. Сердюк  \rm (Інститут математики НАН України, Київ) \\
\noindent \bf Т.А. Степанюк  \rm (Інститут математики НАН України, Київ)

\noindent {\bf A.S. Serdyuk} (Institute of Mathematics NAS of
Ukraine, Kyiv) \\
 \noindent {\bf T.A. Stepaniuk} (Institute of Mathematics NAS of
Ukraine, Kyiv)\\

 \vskip 5mm

{\bf Рiвномiрнi наближення сумами Фур’є на множинах згорток перiодичних функцiй високої гладкостi}

\vskip 5mm

{\bf Uniform approximations  
by  Fourier sums on the sets  of  convolutions of periodic functions of high smoothness}

\vskip 5mm

 \rm На множинах $2\pi$--періодичних функцій $f$, котрі задаються $(\psi, \beta)$--інтегралами від функцій $\varphi$ із $L_{1}$ встановлено нерівності типу Лебега, в яких рівномірні норми відхилень сум Фур'є виражаються через найкращі  наближення в середньому тригонометричними поліномами функцій  $\varphi$. Доведено асимптотичну непокращуваність  одержаних оцінок за умови, коли послідовності $\psi(k)$ спадають до нуля швидше за довільну степеневу функцію.
 
 В ряді важливих випадків встановлено асимптотичні рівності для точних верхніх меж рівномірних наближень сумами Фур'є на класах $(\psi, \beta)$--інтегралів  від функцій $\varphi$, що належать одиничній кулі з простору $L_{1}$.
 

\vskip 5mm

 \rm On the sets of $2\pi$--periodic functions $f$, which are defined with a help of $(\psi, \beta)$--integrals of the functions $\varphi$ from $L_{1}$, we establish Lebesgue-type inequalities, in which the uniform norms of deviations of Fourier sums are expressed via the best approximations by trigonometric polynomials of the functions  $\varphi$. We prove that obtained estimates are best possible, in the case when the sequences $\psi(k)$ decrease to zero faster than any power function. 
 
 In some important cases we establish the asymptotic equalities for the exact upper boundaries of uniform approximations by Fourier sums on the classes of  $(\psi, \beta)$--integrals of the functions  $\varphi$, which belong to the unit ball of the space  $L_{1}$.
 
\newpage


\section{Вступ}

Нехай $L_{1}$
 --- простір $2\pi$--періодичних сумовних на
  $[0,2\pi)$ функцій $f$ в якому норма задається формулою
$\|f\|_{1}=\int\limits_{0}^{2\pi}|f(t)|dt$; $L_{\infty}$ --- простір вимірних і суттєво обмежених     $2\pi$--періодичних функцій $f$ з нормою
$\|f\|_{\infty}=\mathop{\rm{ess}\sup}\limits_{t}|f(t)|$; $C$ --- простір неперервних  $2\pi$--періодичних функцій $f$, в якому норма означається рівністю
 ${\|f\|_{C}=\max\limits_{t}|f(t)|}$.

 
Нехай $\psi(k)$ --- довільна фіксована послідовність дійсних невід'ємних чисел, і  нехай  $\beta$ --- фіксоване дійсне число. 
Позначимо через $C^{\psi}_{\beta}L_{1}$ множину $2\pi$--періодичних функцій, які при всіх $x\in\mathbb{R}$ зображуються у вигляді згортки
\begin{equation}\label{conv}
f(x)=\frac{a_{0}}{2}+\frac{1}{\pi}\int\limits_{-\pi}^{\pi} \Psi_{\beta}(x-t)\varphi(t)dt,
\ a_{0}\in\mathbb{R}, \ \varphi\in L_{1}, \  \varphi\perp1 \
\end{equation}
з твірним ядром $\Psi_{\beta}$ вигляду
\begin{equation}\label{kernelPsi}
\Psi_{\beta}(t)=\sum\limits_{k=1}^{\infty}\psi(k)\cos
\big(kt-\frac{\beta\pi}{2}\big), \ \psi(k)\geq 0, \  \beta\in
    \mathbb{R},
\end{equation}
таким, що
 \begin{equation}\label{condition}
\sum\limits_{k=1}^{\infty}\psi(k)<\infty.
\end{equation}

Якщо функції $f$ і $\varphi$ пов'язані рівністю (\ref{conv}), то функцію $f$ в цьому співвідношенні називають $(\psi,\beta)$--похідною функції $f$ і позначають через $f^{\psi}_{\beta}$. З іншого боку функцію $f$ у рівності  (\ref{conv}) називають $(\psi, \beta)$--інтегралом функції $\varphi$ і позначають через $\mathcal{J}^{\psi}_{\beta}\varphi$. Поняття $(\psi, \beta)$--похідної  ($(\psi, \beta)$--інтеграла) введені О.І. Степанцем  \cite{Stepanets1986_1}, \cite{Step monog 1987}, \cite{Stepanets1}.

Підмножину функцій $f$ з $C^{\psi}_{\beta}L_{1}$ таких, що $f^{\psi}_{\beta}\in B_{1}$,
де $B_{1}$ --- одинична куля в просторі $L_{1}$, тобто
\begin{equation*}
B_{1}:=\left\{\varphi: \ ||\varphi||_{1}\leq 1\right\},
\end{equation*}
будемо позначати через c $C^{\psi}_{\beta,1}$. Зрозуміло, що умова (\ref{condition}) гарантує неперервність твірного ядра $\Psi_{\beta}(t)$ вигляду (\ref{kernelPsi}), а отже і істинність вкладення 
 $C^{\psi}_{\beta}L_{1}\subset C \ \ (C^{\psi}_{\beta,1}\subset C)$.

У випадку, коли $\psi(k)=e^{-\alpha k^{r}}$, $\alpha>0$, $r>0$, ядра $\Psi_{\beta}(t)$ вигляду (\ref{kernelPsi}) є узагальненими ядрами Пуассона, тобто $\Psi_{\beta}(t)=P_{\alpha,r,\beta}(t)$, де
\begin{equation}\label{kernelPsi_GeneralizedPoisson}
P_{\alpha, r, \beta}(t)=\sum\limits_{k=1}^{\infty}\psi(k)\cos
\big(kt-\frac{\beta\pi}{2}\big), \ \alpha> 0, \ r>0, \ \beta\in\mathbb{R}.
\end{equation}
При цьому множини $C^{\psi}_{\beta}L_{1}$ та $C^{\psi}_{\beta,1}$ позначатимемо відповідно через  $C^{\alpha,r}_{\beta}L_{1}$ та  $C^{\alpha,r}_{\beta,1}$ і
називатемо множинами узагальнених інтегралів Пуассона, а відповідні $(\psi,\beta)$--похідні $f^{\psi}_{\beta}$  та $(\psi,\beta)$--інтеграли  $\mathcal{J}^{\psi}_{\beta}\varphi$ позначатимемо через $f^{\alpha,r}_{\beta}$  та  $\mathcal{J}^{\alpha,r}_{\beta}\varphi$ відповідно.

Простір усіх тригонометричних поліномів $t_{n-1}$ порядку не вищого за $n-1$ будемо позначати через  $\mathcal{T}_{2n-1}$.
Нехай $E_{n}(f)_{L_{1}}$ --- найкращі наближення в середньому  тригонометричними поліномами $t_{n-1}\in \mathcal{T}_{2n-1}$, тобто
\begin{equation*}
E_{n}(f)_{L_{1}}=\inf\limits_{t_{n-1}\in \mathcal{T}_{2n-1}}\|f-t_{n-1}\|_{1}.
\end{equation*}

Позначимо через $\rho_{n}(f;x)$ відхилення від  функції $f$ з $L_1$ її частинної суми Фур'є $S_{n-1}(f;\cdot)$ порядку  $n-1$
 \begin{equation}\label{rhoF}
\rho_{n}(f;x):=f(x)-S_{n-1}(f;x).
\end{equation}

 Норми    $\|\rho_{n}(f;\cdot)\|_{C}$ можна оцінити зверху через найкращі рівномірні наближення $E_{n}(f)_{C}=\inf\limits_{t_{n-1}\in \mathcal{T}_{2n-1}}\|f-t_{n-1}\|_{C}$  за допомогою нерівності Лебега
\begin{equation}\label{LebeqIneq}
\| \rho_{n}(f; \cdot) \|_{C} \leq (1+ L_{n-1})E_{n}(f)_{C}, \ n\in\mathbb{N}, \ f\in C,
\end{equation}
де величини $L_{n-1}$  --- константи Лебега сум Фур'є 
\begin{equation*}
L_{n-1}=\frac{1}{\pi}\int\limits_{-\pi}^{\pi}|D_{n-1}(t)|dt=
\frac{2}{\pi}\int\limits_{0}^{\frac{\pi}{2}}  \frac{|\sin (2n-1)t|}{\sin t} dt,
\end{equation*}
\begin{equation*}
D_{n-1}(t):=
\frac{1}{2}+\sum\limits_{k=1}^{\infty}\cos kt=\frac{\sin(n-\frac{1}{2})t}{2\sin \frac{t}{2}}.
\end{equation*}

При цьому, як встановив Фейєр \cite{Fejer},  для констант Лебега  
$L_{n}$ має місце асимптотична рівність
 \begin{equation}\label{Fejer_Ln}
L_{n}=\frac{4}{\pi^{2}} \ln n+ \mathcal{O}(1), \ \ n\rightarrow\infty,
\end{equation} 
де $\mathcal{O}(1)$ --- рівномірно обмежена по $n$ величина.

Більш точні оцінки для різниць  $L_{n}-\frac{4}{\pi^{2}} \ln (n+a)$, $a>0$, при $n\in\mathbb{N}$ можна знайти в роботах \cite{Akhiezer}, \cite{Dzyadyk}, \cite{Galkin}, \cite{Natanson}, \cite{Shakirov}, \cite{ZhukNatanson} та ін.

З урахуванням \eqref{Fejer_Ln}, нерівність \eqref{LebeqIneq} можна записати у вигляді 
\begin{equation}\label{LebeqIneq0}
\| \rho_{n}(f; \cdot) \|_{C} \leq \left(\frac{4}{\pi^{2} }\ln n + \mathcal{O}(1)\right)E_{n}(f)_{C},  \ \ f\in C.
\end{equation}

Незважаючи на загальність, нерівність (\ref{LebeqIneq0}) на всьому просторі $C$ є точною за порядком. Більш того, вона є асимптотично непокращуваною в тому сенсі, що константа $\frac{4}{\pi^{2}}$ у формулі \eqref{LebeqIneq0} зменшена бути не може.

Разом з тим, використання нерівностей \eqref{LebeqIneq} і \eqref{LebeqIneq0} для функцій $f$ із функціональних множин $C^{\psi}_{\beta}L_{1}$ чи $C^{\psi}_{\beta, 1}$ може виявитись неефективним.
Більш того, існують послідовності $\psi$ такі, що для  $f\in C^{\psi}_{\beta}L_{1}$ зазначені нерівності  є неточними навіть за порядком.
Щоб у цьому переконатись, покладемо $\psi(k)=e^{-\alpha k}$ і розглянемо породжені такими послідовностями класи $C^{\psi}_{\beta,  1}=C^{\alpha,1}_{\beta, 1}$. Як показано в \cite{Serdyuk2005} при всіх $\alpha>0$, $\beta\in\mathbb{R}$ справедлива асимптотична при $n\rightarrow\infty$ рівність
\begin{equation}\label{Serdyuk_Asymp_r=1}
 {\cal E}_{n}(C^{\alpha,1}_{\beta,1})_{C}=\sup\limits_{f\in
C^{\alpha,1}_{\beta,1}}\| \rho_{n}(f;\cdot)\|_{C}
=e^{-\alpha n}\left(\frac{1}{\pi(1-e^{-\alpha}) }+ \frac{\mathcal{O}(1)}{n}\frac{e^{-\alpha}}{ (1-e^{-\alpha})^{2} } \right),
  \end{equation}
в якій $\mathcal{O}(1)$ --- рівномірно обмежена по $\alpha$, $\beta$, $n$ величина.

Крім того, (див., наприклад, \cite[c. 48]{Stepanets2}) для найкращих наближень $ {E}_{n}(C^{\alpha,1}_{\beta,1})_{C}$ справедливі точні за порядком оцінки
\begin{equation}\label{Stepanets_BestApprox_r=1}
 K^{(1)} e^{-\alpha n}\leq {E}_{n}(C^{\alpha,1}_{\beta,1})_{C} \leq K^{(2)} e^{-\alpha n},
  \end{equation}
в яких $ K^{(1)}$ i $ K^{(2)}$ --- деякі додатні сталі.

Тоді для $f\in C^{\alpha,1}_{\beta,1} $ в силу \eqref{Serdyuk_Asymp_r=1} виконується нерівність
\begin{equation}\label{Rho_Ineq_r=1}
 \| \rho_{n}(f;\cdot)\|_{C}
\leq e^{-\alpha n}\left(\frac{1}{\pi(1-e^{-\alpha}) }+ \frac{\mathcal{O}(1)}{n}\frac{e^{-\alpha}}{ n(1-e^{-\alpha})^{2} } \right),
  \end{equation}
  в той час як використання класичної нерівності Лебега  та оцінки \eqref{LebeqIneq0} дозволяє записати більш грубу за порядком оцінку
  \begin{equation*}
\| \rho_{n}(f; \cdot) \|_{C} \leq e^{-\alpha n} \left(\frac{4K^{(2)}}{\pi^{2} }\ln n + \mathcal{O}(1)\right).
\end{equation*}

У роботі  \cite{Stepanets1989N4} О.І. Степанець, розглядаючи послідовності $\psi(k)$, що спадають до нуля повільніше за  будь-яку геометричну прогресію, встановив аналоги нерівностей Лебега для множин $(\psi,\beta)$--диференційовних функцій $C^{\psi}_{\beta}C \subset C^{\psi}_{\beta}L_{1}$, ($C^{\psi}_{\beta}C$ --- множина $2\pi$--періодичних функцій $f(x)$, які при всіх $x\in\mathbb{R}$ зображуються у вигляді \eqref{conv}, де $\varphi \in C$) в яких норми відхилень $\| \rho_{n}(f;\cdot)\|_{C}$ виражаються через найкращі наближення $E_{n}(f^{\psi}_{\beta})_{C}$. Одержані в \cite{Stepanets1989N4}  нерівності виявились асимптотично точними не тільки на всіх множинах $C^{\psi}_{\beta}C$, але і на деяких важливих підмножинах із $C^{\psi}_{\beta}C$, зокрема  на класах $C^{\psi}_{\beta}C^{0}= \left\{ f\in C^{\psi}_{\beta}C: \|f^{\psi}_{\beta}\|_{C}\leq 1 \right\}$. Згодом дослідження по встановленню асимптотично точних нерівностей типу Лебега на множинах $(\psi,\beta)$--диференційовних функцій були продовжені в роботах 
\cite{StepanetsSerdyuk2000}, \cite{MusienkoSerdyuk2013_4}, 
\cite{MusienkoSerdyuk2013_5},  \cite{SerdyukMusienko2010},
\cite{StepanetsSerdyuk2000No3}, 
\cite{SerdyukStepanyukFilomat}, \cite{SerdyukStepanyukJAEN}. 

В даній роботі буде знайдено нерівності типу Лебега на множинах $C^{\psi}_{\beta}L_{1}$, в яких норми  $\|\rho_{n}(f;x)\|_{C}$ виражаються через   $E_{n}(f^{\psi}_{\beta})_{L_{1}}$ і доведено їх асимптотичну непокращуваність у випадку, коли
 \begin{equation}\label{LimitCase_BestPossibility}
 \lim\limits_{n\rightarrow\infty}\frac{  \frac{1}{n}\sum\limits_{k=1}^{\infty}k\psi(k+n)}{\sum\limits_{k=n}^{\infty}\psi(k)}=0.
  \end{equation}

Також у роботі, за умови \eqref{LimitCase_BestPossibility} знайдено розв'язок задачі Колмогорова--Нікольського для сум Фур'є на класах $C^{\psi}_{\beta,1}$, яка полягає у відшуканні асимптотичних рівностей величин
 \begin{equation}\label{FourierSum}
 {\cal E}_{n}(C^{\psi}_{\beta,1})_{C}=\sup\limits_{f\in
C^{\psi}_{\beta,1}}\|f(\cdot)-S_{n-1}(f;\cdot)\|_{C}.
  \end{equation}
  
  Проблеми, пов'язані зі знаходженням розв'язку задачі Колмогорова--Нікольського для сум Фур'є на класах згорток досліджувались в роботах \cite{Kol}, \cite{Nikolsky 1946}, \cite{Stechkin 1980},
\cite{Teljakovsky1968}, \cite{Teljakovsky1989}, \cite{Stepanets1986_1}, \cite{Stepanets1}, \cite{Serdyuk2005}, \cite{Serdyuk2005Lp}, \cite{SerdyukStepanyuk2018}, \cite{Telyakovskiy1961}, \cite{SerdyukSokolenkoMFAT2019}, \cite{SerdyukSokolenkoUMJ2022} та ін.

\section{Основні результати}

Має місце наступне твердження.

\begin{theorem}\label{theorem2}
Нехай $\sum\limits_{k=1}^{\infty}k\psi(k)<\infty$, $\psi(k)\geq 0$, $k=1,2,...$, $\beta\in\mathbb{R}$ i $n\in\mathbb{N}$. Тоді для довільної функції
 $f\in C^{\psi}_{\beta}L_{1}$   має місце  нерівність
\begin{equation}\label{Theorem2Ineq1}
\|f(\cdot)-S_{n-1}(f;\cdot)\|_{C}\leq
 \frac{1}{\pi}
\sum\limits_{k=n}^{\infty}\psi(k)
 E_{n}(f^{\psi}_{\beta})_{L_{1}}.
 \end{equation}

 Крім того, для довільної функції $f\in  C^{\psi}_{\beta}L_{1}$ можна знайти функцію  ${\mathcal{F}(x)=\mathcal{F}(f;n,x)}$ з множини $ C^{\psi}_{\beta}L_{1}$ таку, що
  $E_{n}(\mathcal{F}^{\psi}_{\beta})_{L_{1}}=E_{n}(f^{\psi}_{\beta})_{L_{1}}$ 
і  має місце  рівність
 \begin{equation}\label{Theorem2Eq}
\|\mathcal{F}(\cdot)-S_{n-1}(\mathcal{F};\cdot)\|_{C}= \left( \frac{1}{\pi}
\sum\limits_{k=n}^{\infty}\psi(k)+
 \frac{\xi}{n} \sum\limits_{k=1}^{\infty}k\psi(k+n)\right)E_{n}(f^{\psi}_{\beta})_{L_{1}}.
 \end{equation}  
В (\ref{Theorem2Eq}) величина  $\xi=\xi(f;n;\psi;\beta)$ є такою, що $-2\leq \xi\leq 0$.
\end{theorem}

\begin{proof}[Доведення Теореми~\ref{theorem2}]

Нехай $f\in C^{\psi}_{\beta}L_{1}$. Тоді згідно з (\ref{conv}) в кожній точці $x\in \mathbb{R}$ має місце  інтегральне представлення
\begin{equation}\label{repr}
\rho_{n}(f;x)=f(x)-S_{n-1}(f;x)=
\frac{1}{\pi}\int\limits_{-\pi}^{\pi}f^{\psi}_{\beta}(t)\Psi_{\beta,n}(x-t)dt,
\end{equation}
де
 \begin{equation}\label{kernelN}
\Psi_{\beta,n}(t):=
\sum\limits_{k=n}^{\infty}\psi(k)\cos\Big(kt-\frac{\beta\pi}{2}\Big),  \  \beta\in\mathbb{R}.
\end{equation}
 
Функція $\Psi_{\beta,n}(t)$ ортогональна до будь--якого тригонометричного полінома  $t_{n-1}$ порядку не вищого за $n-1$. Тоді. для довільного полінома $t_{n-1} \in \mathcal{T}_{2n-1}$ отримаємо
\begin{equation}\label{for1}
\rho_{n}(f;x)=
\frac{1}{\pi}\int\limits_{-\pi}^{\pi}\delta_{n}(t)\Psi_{\beta,n}(x-t)dt,
\end{equation}
де
\begin{equation}\label{delta}
\delta_{n}(\cdot)=\delta_{n}(\psi,\beta,t_{n-1};\cdot):=f^{\psi}_{\beta}(\cdot)-t_{n-1}(\cdot).
\end{equation}

Виберемо в якості $t_{n-1}$ в формулі (\ref{for1})  поліном $t_{n-1}^{*}$ найкращого наближення функції  $f^{\psi}_{\beta}$ в просторі  $L_{1}$, тобто такий що
$$
\| f^{\psi}_{\beta}-t^{*}_{n-1}\|_{1}=E_{n}(f^{\psi}_{\beta})_{L_{1}}, \ \ 
$$
 Тоді, використовуючи нерівність
\begin{equation}\label{HolderIneq}
\bigg\|\int\limits_{-\pi}^{\pi}K(t-u)\varphi(u)du \bigg\|_{C}\leq \|K\|_{p'}\|\varphi\|_{p},
\end{equation}
$$
\varphi\in L_{p}, \ \ K\in L_{p'}, \ \ 1\leq p\leq\infty, \ \ \frac{1}{p}+\frac{1}{p'}=1
$$
(див., наприклад, \cite[c. 43]{Korn}), отримуємо
\begin{equation}\label{for2}
\|f(\cdot)-S_{n-1}(f;\cdot)\|_{C}
\leq \left\|
\frac{1}{\pi}\int\limits_{-\pi}^{\pi}  (f^{\psi}_{\beta}(t)-t^{*}_{n-1}(t))  \Psi_{\beta,n}(\cdot-t)dt
\right\|_{C}
\leq
\frac{1}{\pi}\|\Psi_{\beta,n}\|_{C}E_{n}(f^{\psi}_{\beta})_{L_{1}}.
\end{equation}

Знайдемо двосторонні оцінки норми $\|\Psi_{\beta,n}\|_{C}$.
Покажемо, що при всіх  $n\in\mathbb{N}$ i $\beta\in\mathbb{R}$ справедлива оцінка
\begin{equation}\label{f11}
\sum\limits_{k=n}^{\infty}\psi(k)-
\frac{\pi}{n} \sum\limits_{k=1}^{\infty}k\psi(k+n)
\leq
\|\Psi_{\beta,n}\|_{C}\leq
\sum\limits_{k=n}^{\infty}\psi(k).
\end{equation}

Оцінка зверху для $\|\Psi_{\beta,n}\|_{C}$ в (\ref{f11}) випливає безпосередньо з (\ref{kernelN}).

 Для оцінки $\|\Psi_{\beta,n}\|_{C}$ знизу   представимо функцію $ \Psi_{\beta,n}(t)$, яка означена формулою  (\ref{kernelN}), у вигляді
\begin{equation}\label{pp}
  \Psi_{\beta,n}(t)=
g_{\psi,n}(t)\cos\Big(nt-\frac{\beta\pi}{2}\Big)+h_{\psi,n}(t)\sin\Big(nt-\frac{\beta\pi}{2}\Big),
\end{equation}
де
\begin{equation}\label{g}
  g_{\psi,n}(t):=
\sum\limits_{k=0}^{\infty}\psi(k+n)\cos kt,
\end{equation}
\begin{equation}\label{h}
 h_{\psi,n}(t):=
-\sum\limits_{k=0}^{\infty}\psi(k+n)\sin kt.
\end{equation}

Оскільки величина  $\|\Psi_{\beta,n}\|_{C}$ періодична з періодом 2 за параметром $\beta$, то, не зменшуючи загальності, можна вважати, що $\beta\in [0,2]$.

Позначимо
\begin{equation}\label{t0}
t_{0}:=\frac{\beta\pi}{2n}, \ \ \beta\in [0,2].
\end{equation}

В силу (\ref{pp})
\begin{equation}\label{Psi_t0}
  \Psi_{\beta,n}(t_{0})=
g_{\psi,n}(t_{0}).
\end{equation}

Тоді

\begin{align}\label{Psi_t0_1}
&\|\Psi_{\beta,n}\|_{C} \geq | \Psi_{\beta,n}(t_{0})|
=|g_{\psi,n}(t_{0})|
=|g_{\psi,n}(0)+ (g_{\psi,n}(t_{0})-g_{\psi,n}(0))|
 \notag \\
&\geq
|g_{\psi,n}(0)|-\left| (g_{\psi,n}\left(\frac{\beta\pi}{2n}\right)-g_{\psi,n}(0)) \right|
=\sum\limits_{k=n}^{\infty}\psi(k)-
\left| (g_{\psi,n}\left(\frac{\beta\pi}{2n}\right)-g_{\psi,n}(0)) \right|.
\end{align}

Використовуючи теорему про середнє, маємо
\begin{equation}\label{Psi_t0_2}
\left| g_{\psi,n}\left(\frac{\beta\pi}{2n}\right)-g_{\psi,n}(0) \right|
\leq
\| g_{\psi,n}^{'} \|_{C}\frac{\beta\pi}{2n}\leq 
\frac{\pi}{n} \sum\limits_{k=1}^{\infty}k\psi(k+n).
\end{equation}

Із  (\ref{Psi_t0_1}) i  (\ref{Psi_t0_2}) випливає шукана оцінка знизу для норм $\|\Psi_{\beta,n}\|_{C}$ в співвідношенні (\ref{f11})
\begin{equation}\label{f11_2}
\|\Psi_{\beta,n}\|_{C}\geq
\sum\limits_{k=n}^{\infty}\psi(k)-
\frac{\pi}{n} \sum\limits_{k=1}^{\infty}k\psi(k+n).
\end{equation}

Оцінку (\ref{f11}) можна записати у вигляді формули
\begin{equation}\label{f11_3}
\|\Psi_{\beta,n}\|_{C}=
\sum\limits_{k=n}^{\infty}\psi(k)+
\frac{\Theta_{1}\pi}{n} \sum\limits_{k=1}^{\infty}k\psi(k+n),
\end{equation}
де для величини $\Theta_{1}=\Theta_{1}(n,\beta,\psi)$ виконуються нерівності
\begin{equation}\label{theta}
-1\leq \Theta_{1}\leq 0.
\end{equation}

Отже, із (\ref{for2}) і (\ref{f11_3}) випливає нерівність (\ref{Theorem2Ineq1}) з  Теореми~\ref{theorem2}.

Доведемо другу частину Теореми~\ref{theorem2}. Для цього нам необхідно для довільної функції  $\varphi\in L_{1}$ знайти функцію $\Phi(\cdot)=\Phi(\varphi, \cdot)\in L_{1}$ таку, що
$E_{n}(\Phi)_{L_{1}}=E_{n}(\varphi)_{L_{1}}$ і для якої  виконується рівність
\begin{align}\label{th2Eq1}
&\frac{1}{\pi}\left|\int\limits_{-\pi}^{\pi}\left(\Phi(t) -t_{n-1}^{*}(t)\right)\Psi_{\beta,n}(0-t)dt\right| \notag \\
&= \left( \frac{1}{\pi}
\sum\limits_{k=n}^{\infty}\psi(k)+
 \frac{\xi}{n} \sum\limits_{k=1}^{\infty}k\psi(k+n)\right)E_{n}(\varphi)_{L_{1}},
\end{align}
де $t_{n-1}^{*}$ --- поліном найкращого наближення порядку   $n-1$ функції  $\Phi$ в просторі $L_{1}$.

В цьому випадку для функції $f\in C^{\psi}_{\beta}L_{1}$ існує функція $\Phi(\cdot)=\Phi(f^{\psi}_{\beta};\cdot)$ така, що $E_{n}(\Phi)_{L_{1}}=E_{n}(f^{\psi}_{\beta})_{L_{1}}$, і має місце формула  (\ref{th2Eq1}), де в ролі  $\varphi$ виступає функція $f^{\psi}_{\beta}$.

Розглянемо функцію
\begin{equation*}
\mathcal{F}(\cdot)=\mathcal{J}^{\psi}_{\beta}(\Phi(\cdot)-\frac{a_{0}}{2}),
\end{equation*}
де
\begin{equation*}
a_{0}=a_{0}(\Phi):=\frac{1}{\pi}\int\limits_{-\pi}^{\pi}\Phi(t)dt.
\end{equation*}

Функція  $\mathcal{F}$ є шуканою функцією, оскільки
$\mathcal{F}\in C^{\psi
}_{\beta}L_{1}$ і
\begin{equation*}
E_{n}(\mathcal{F}^{\psi}_{\beta})_{L_{1}}=E_{n}(\Phi-\frac{a_{0}}{2})_{L_{1}}=
E_{n}(\Phi)_{L_{1}}=E_{n}(f^{\psi}_{\beta})_{L_{1}},
\nonumber
\end{equation*}
і на підставі  (\ref{for1}) і (\ref{th2Eq1}) має місце оцінка (\ref{Theorem2Eq}).

Доведемо (\ref{th2Eq1}). Нехай $t^{*}$ --- точка з  проміжку $T=\Big[\frac{\pi(1-\beta)}{2n}, \ 2\pi+\frac{\pi(1-\beta)}{2n} \Big)$, в якій функція $|\Psi_{-\beta,n}(t)|$  набуває свого найбільшого значення, тобто,
\begin{equation*}
|\Psi_{-\beta,n}(t^{*})|=\| \Psi_{-\beta,n}\|_{C}=
\| \Psi_{\beta,n}\|_{C}.
\nonumber
\end{equation*}

Покладемо $\Delta_{k}^{n}:=\Big[\frac{(k-1)\pi}{n}+\frac{\pi(1-\beta)}{2n}, \frac{k\pi}{n}+\frac{\pi(1-\beta)}{2n} \Big)$, $k=1,...,2n$.
Через $k^{*}$   позначимо номер такий, що $t^{*}\in \Delta_{k^{*}}^{n}$.
Оскільки функція $\Psi_{-\beta,n}$ є абсолютно неперервною, то для довільного $\varepsilon>0$ існує сегмент $\ell^{*}=[\xi^{*}, \xi^{*}+\delta]\subset \Delta_{k^{*}}^{n}$ такий, що для довільного  $t\in \ell^{*}$ виконується нерівність
${|\Psi_{\beta,n}(t)|>\| \Psi_{\beta,n}\|_{C}-\varepsilon}$.
Ясно, що $\mathrm{mes}\, \ell^{*}=|\ell^{*}|=\delta<\frac{\pi}{n}$.

Для довільного $\varphi\in L_{1}$ і $\varepsilon>0$ розглянемо функцію $\Phi_{\varepsilon}(t)$, яка на проміжку $T$ означена за допомогою рівностей

\begin{equation}\label{Phi_varepsilon}
\Phi_{\varepsilon}(t)=
\begin{cases}
E_{n}(\varphi)_{L_1}\frac{1-\varepsilon(2\pi-\delta)}{\delta}\mathrm{sign}\cos \Big( nt+\frac{\beta\pi}{2}\Big), & t\in \ell^{*}, \\
E_{n}(\varphi)_{L_1} \varepsilon \  \mathrm{sign}\cos\Big( nt+\frac{\beta\pi}{2}\Big), &
t\in \mathrm{T}\setminus \ell^{*}.
  \end{cases}
\end{equation}

Для функції $\Phi_{\varepsilon}(t)$ при достатньо малих значеннях $\varepsilon>0$  $(\varepsilon\in(0, \frac{1}{2\pi}))$ має місце наступна рівність:
\begin{align}\label{Phi_L1}
 \|\Phi_{\varepsilon}\|_{1}&=
E_{n}(\varphi)_{L_1}\frac{1-\varepsilon(2\pi-\delta)}{\delta}
\int\limits_{\ell^{*}}\Big| \mathrm{sign}\cos\Big( nt+\frac{\beta\pi}{2}\Big) \Big| dt \notag \\
&+
E_{n}(\varphi)_{L_1}\varepsilon
\int\limits_{\mathrm{T}\setminus \ell^{*}}\Big| \mathrm{sign}\cos\Big( nt+\frac{\beta\pi}{2}\Big) \Big| dt
\notag \\
&= 
E_{n}(\varphi)_{L_1}\left(\frac{1-\varepsilon(2\pi-\delta)}{\delta}\delta+\varepsilon(2\pi-\delta)   \right)=E_{n}(\varphi)_{L_1}.
\end{align}

Крім того, згідно з \eqref{Phi_varepsilon}
\begin{equation}\label{sign}
\mathrm{sign} \, \Phi_{\varepsilon}(t)=\mathrm{sign}\cos\Big( nt+\frac{\beta\pi}{2}\Big).
\end{equation}

Оскільки для довільного тригонометричного полінома $t_{n-1}\in \mathcal{T}_{2n-1}$
\begin{equation*}
\int\limits_{0}^{2\pi}t_{n-1}(t)\mathrm{sign}\cos\Big( nt+\frac{\beta\pi}{2}\Big) dt=0,
\nonumber
\end{equation*}
то, з урахуванням (\ref{sign}), виконується рівність
\begin{equation*}
\int\limits_{0}^{2\pi}t_{n-1}(t)\mathrm{sign}\Big(\Phi_{\varepsilon}(t)-0 \Big) dt=0,
\ \ \ t_{n-1}\in\mathcal{T}_{2n-1}.
\nonumber
\end{equation*}

Згідно з Теоремою 1.4.5 роботи \cite[с.28]{Korn}, поліном $t_{n-1}^{*}\equiv0$ є поліномом найкращого наближення функції  $\Phi_{\varepsilon}$ в метриці простору $L_{1}$, тобто, $E_{n}(\Phi_{\varepsilon})_{L_1}=\| \Phi_{\varepsilon}\|_{1}$, отже з (\ref{Phi_L1}) випливає
$E_{n}(\Phi_{\varepsilon})_{L_1}=E_{n}(\varphi)_{L_1}$.

Крім того, для функції $\Phi_{\varepsilon}$
\begin{align}\label{form_eq1}
&\frac{1}{\pi}\int\limits_{-\pi}^{\pi}(\Phi_{\varepsilon}(t)-t_{n-1}^{*}(t))\Psi_{\beta,n}(-t)dt
 =
\frac{1}{\pi}\int\limits_{-\pi}^{\pi}\Phi_{\varepsilon}(t)\Psi_{-\beta,n}(t)dt \notag \\
=&\frac{1-\varepsilon(2\pi-\delta)}{\pi\delta}
E_{n}(\varphi)_{L_1}
\int\limits_{\ell^{*}} \mathrm{sign}\cos\Big( nt+\frac{\beta\pi}{2}\Big) \Psi_{-\beta,n}(t) dt \notag \\
+&
\frac{\varepsilon}{\pi}E_{n}(\varphi)_{L_1}
\int\limits_{\mathrm{T}\setminus \ell^{*}} \mathrm{sign}\cos\Big( nt+\frac{\beta\pi}{2}\Big) \Psi_{-\beta,n}(t) dt.
\end{align}

Враховуючи, що $ \mathrm{sign} \, \Phi_{\varepsilon}(t)=(-1)^{k}$, $t\in \Delta_{k}^{(n)}, \  k=1,..., 2n$, а також вкладення $\ell^{*}\subset \Delta_{k^{*}}^{(n)}$, отримуємо
\begin{align}\label{form_eq2}
&\left|\frac{1-\varepsilon(2\pi-\delta)}{\pi\delta}
E_{n}(\varphi)_{L_1}
\int\limits_{\ell^{*}} \mathrm{sign}\cos\Big( nt+\frac{\beta\pi}{2}\Big) \Psi_{-\beta,n}(t) dt \right| \notag \\
=
&\left|(-1)^{k^{*}}\frac{1-\varepsilon(2\pi-\delta)}{\pi\delta}
E_{n}(\varphi)_{L_1}
\int\limits_{\ell^{*}}  \Psi_{-\beta,n}(t) dt \right| \notag \\
\geq& 
\frac{1-\varepsilon(2\pi-\delta)}{\pi}
E_{n}(\varphi)_{L_1}
\left( \|\Psi_{\beta,n}\|_{C}-\varepsilon\right)
\notag \\
>&
\frac{1-2\pi\varepsilon}{\pi}
E_{n}(\varphi)_{L_1}
\left( \|\Psi_{\beta,n}\|_{C}-\varepsilon\right) \notag \\
=& 
\frac{1}{\pi}
E_{n}(\varphi)_{L_1}
\left( \|\Psi_{\beta,n}|_{C}-
2\pi\varepsilon \|\Psi_{\beta,n}\|_{C}-\varepsilon+2\pi\varepsilon^{2}\right)\notag \\
>&
E_{n}(\varphi)_{L_1} \left(\frac{1}{\pi} \|\Psi_{\beta,n}\|_{C}-
\varepsilon \Big(2 \|\Psi_{\beta,n}\|_{C}+\frac{1}{\pi}\Big) \right).
\end{align}

Крім того, неважко переконатись, що
\begin{align}\label{form_eq3}
&\left| \frac{\varepsilon}{\pi}E_{n}(\varphi)_{L_1}
\int\limits_{\mathrm{T}\setminus \ell^{*}} \mathrm{sign}\cos\Big( nt+\frac{\beta\pi}{2}\Big) \Psi_{-\beta,n}(t) dt\right| \notag \\
\leq & \frac{\varepsilon}{\pi}E_{n}(\varphi)_{L_1} \| \Psi_{\beta,n}\|_{C}(2\pi-\delta)
<2\varepsilon E_{n}(\varphi)_{L_1} \| \Psi_{\beta,n}\|_{C}.
\end{align}

З формул (\ref{form_eq1})--(\ref{form_eq3}) випливає наступна оцінка:
\begin{align}\label{form_eq4}
& \left| \int\limits_{-\pi}^{\pi}\frac{1}{\pi}(\Phi_{\varepsilon}(t)-t_{n-1}^{*}(t)) \Psi_{\beta,n}(-t)dt \right| \notag \\
>&
E_{n}(\varphi)_{L_1}\left(\frac{1}{\pi} \| \Psi_{\beta,n}\|_{C}
-\varepsilon\Big(4 \| \Psi_{\beta,n}\|_{C}+\frac{1}{\pi}\Big)\right).
\end{align}

Виберемо  $\varepsilon$ настільки малим, щоб
\begin{equation}\label{form_varepsilon}
\varepsilon< \frac{\pi \sum\limits_{k=1}^{\infty}k \psi(k+n)}{n\left(1+4\pi \sum\limits_{k=n}^{\infty} \psi(k)\right)}
\end{equation}
і для цього
 $\varepsilon$ покладемо
\begin{equation}\label{form_eq13}
\Phi(t)=\Phi_{\varepsilon}(t).
\end{equation}

Функція $\Phi(t)$ є шуканою функцією, оскільки 
 ${E_{n}(\Phi)_{L_1}=E_{n}(\varphi)_{L_1}}$ і згідно з (\ref{f11}), (\ref{form_eq4})  i (\ref{form_varepsilon}) 
 \begin{align}\label{form_eq14}
&\left| \frac{1}{\pi} \int\limits_{-\pi}^{\pi}(\Phi(t)-t_{n-1}^{*}(t))\Psi_{\beta,n}(-t)dt\right| \notag \\
\geq
& \left( \frac{1}{\pi}
\sum\limits_{k=n}^{\infty}\psi(k)-
 \frac{1}{n} \sum\limits_{k=1}^{\infty}k\psi(k+n) - \varepsilon\left(4 \sum\limits_{k=n}^{\infty} \psi(k)+ \frac{1}{\pi}\right)    \right) E_{n}(\varphi)_{L_1} \notag \\
 \geq
& \left( \frac{1}{\pi}
\sum\limits_{k=n}^{\infty}\psi(k)-
 \frac{1}{n} \sum\limits_{k=1}^{\infty}k\psi(k+n) 
 - 
 \frac{\pi \sum\limits_{k=1}^{\infty}k \psi(k+n)}{n\left(1+4\pi \sum\limits_{k=n}^{\infty} \psi(k)\right)}
\left(4 \sum\limits_{k=n}^{\infty} \psi(k)+ \frac{1}{\pi}\right)    \right) E_{n}(\varphi)_{L_1} \notag \\
\geq
 & \left( \frac{1}{\pi}
\sum\limits_{k=n}^{\infty}\psi(k)-
 \frac{2}{n} \sum\limits_{k=1}^{\infty}k\psi(k+n) \right) E_{n}(\varphi)_{L_1}.
\end{align}

 З формул (\ref{form_eq14}), (\ref{for2}) і (\ref{f11})  випливає (\ref{th2Eq1}).
Теорему~\ref{theorem2} доведено.
\end{proof}


Зрозуміло, що формулу \eqref{Theorem2Ineq1} теореми \ref{theorem2} можна записати однотипно з  формулою \eqref{Theorem2Eq} у наступному вигляді:

\begin{equation}\label{Theorem2IneqCase1}
\|f(\cdot)-S_{n-1}(f;\cdot)\|_{C}\leq
\left(  \frac{1}{\pi}
\sum\limits_{k=n}^{\infty}\psi(k)+
 \frac{\Theta_{1}}{n} \sum\limits_{k=1}^{\infty}k\psi(k+n) \right)
 E_{n}(f^{\psi}_{\beta})_{L_{1}},
 \end{equation}
де $\Theta_{1}=\Theta_{1}(n,\beta,\psi)$ задовольняє нерівності \eqref{theta}.

\begin{theorem}\label{theorem1}
Нехай  $\sum\limits_{k=1}^{\infty}k\psi(k)<\infty$, $\psi(k)\geq 0$, $k=1,2,...$ і  $\beta\in\mathbb{R}$. Тоді при усіх $n\in\mathbb{N}$ має місце формула
\begin{equation}\label{Theorem1Asymp}
{\cal E}_{n}(C^{\psi}_{\beta,1})_{C}=
\frac{1}{\pi}
\sum\limits_{k=n}^{\infty}\psi(k)+
 \frac{\Theta_{2}}{n} \sum\limits_{k=1}^{\infty}k\psi(k+n),
 \end{equation}
де  для величини $\Theta_{2}=\Theta_{2}(n,\beta,\psi)$ виконуються нерівності $-1\leq \Theta_{2} \leq 0$. 
\end{theorem}


\begin{proof} 
Згідно з  \eqref{FourierSum} i \eqref{repr}  отримуємо, що
\begin{equation}\label{f1}
{\cal E}_{n}(C^{\psi}_{\beta,1})_{C}=
\frac{1}{\pi}\sup\limits_{\varphi\in B_{1}^{0}}\bigg\|\int\limits_{-\pi}^{\pi} \varphi(t)\Psi_{\beta,n}(x-t)dt\bigg\|_{C}, 
\end{equation}
де $\Psi_{\beta,n}(\cdot)$ означена рівністю (\ref{kernelN}), а $B_{1}^{0}:=\left\{\varphi \in L_{1}: \ ||\varphi||_{1}\leq 1, \  \varphi\perp1\right\}.$

Беручи до уваги інваріантність множини   $B_{1}^{0}$ відносно зсуву аргументу, з
 (\ref{f1})  отримуємо
\begin{equation}\label{f3}
{\cal E}_{n}(C^{\psi}_{\beta,1})_{C}=
\frac{1}{\pi}\sup\limits_{\varphi\in B_{1}^{0}}\int\limits_{-\pi}^{\pi}  \varphi(t) \Psi_{\beta,n}(t)dt.
\end{equation}

 На основі співвідношення двоїстості (див., напр., \cite{Korn})  маємо
\begin{equation}\label{f4}
\sup\limits_{\varphi\in B_{1}^{0}}\int\limits_{-\pi}^{\pi}\Psi_{\beta,n}(t)\varphi(t)dt=
\inf\limits_{\lambda\in\mathbb{R}}\|\Psi_{\beta,n}(t)-\lambda\|_{C}, \
\end{equation}

Для знаходження двосторонньої оцінки величини $\inf\limits_{\lambda\in\mathbb{R}}\|\Psi_{\beta,n}(t)-\lambda\|_{C}$ нам буде корисним наступне твердження, яке може знайти і самостійне застосування.

\begin{lemma}\label{Lemma_norm}
Нехай $\psi(k)\geq 0$, $\sum\limits_{k=1}^{\infty}k\psi(k) <\infty$. Тоді при всіх $\beta\in\mathbb{R}$ i $n\in\mathbb{N}$  для кожної з величин
\begin{equation}\label{Norm1}
I_{n}^{(1)}
=I_{n}^{(1)}(\psi,\beta):=
\|\Psi_{\beta,n}\|_{C},
\end{equation}
\begin{equation}\label{Norm2}
I_{n}^{(2)}
=I_{n}^{(2)}(\psi,\beta):=
\inf\limits_{\lambda\in\mathbb{R}}\|\Psi_{\beta,n}(t)-\lambda\|_{C},
\end{equation}
\begin{equation}\label{Norm3}
I_{n}^{(3)}
=I_{n}^{(3)}(\psi,\beta):=
\frac{1}{2}\left\|\Psi_{\beta,n}\left(t+\frac{\pi}{n}\right)-\Psi_{\beta,n}(t)\right\|_{C}
\end{equation}
виконуються формули
\begin{equation}\label{Norm_asympt}
I_{n}^{(j)}=
\sum\limits_{k=n}^{\infty}\psi(k)+
 \frac{\Theta_{j}\pi}{n} \sum\limits_{k=1}^{\infty}k\psi(k+n), \ \ \ j=1,2,3,
\end{equation}
в яких для будь-якої з величин $\Theta_{j}=\Theta_{j}(n, \beta,\psi), \ j=1,2,3,$ виконуються двосторонні оцінки
\begin{equation*}
-1\leq \Theta_{j}\leq 0, \ \  \ j=1,2,3.
\end{equation*}
\end{lemma}
\begin{proof}[ Доведення Леми~\ref{Lemma_norm}]

Оскільки
\begin{equation}\label{f5}
\inf\limits_{\lambda\in\mathbb{R}}\|\Psi_{\beta,n}(t)-\lambda\|_{C}
\leq
\|\Psi_{\beta,n}\|_{C}
\end{equation}
і
\begin{equation}\label{f6}
\frac{1}{2}\left\|\Psi_{\beta,n}\left(t+\frac{\pi}{n}\right)-\Psi_{\beta,n}(t)\right\|_{C}
\leq
\inf\limits_{\lambda\in\mathbb{R}}\|\Psi_{\beta, n}(t)-\lambda\|_{C},
\end{equation}
то 
$$I_{n}^{(3)}\leq I_{n}^{(2)} \leq I_{n}^{(1)},
$$
і, отже, необхідна оцінка зверху для кожної з величин $I_{n}^{(j)}, \  j=1,2,3$ випливає з (\ref{f11}).

Залишається знайти оцінку знизу для $I_{n}^{(3)}$. В силу (\ref{pp})--(\ref{h}) i (\ref{Norm3})

\begin{align}\label{f7}
I_{n}^{(3)}=&\frac{1}{2}\left\|\Psi_{\beta,n}\left(t+\frac{\pi}{n}\right)-\Psi_{\beta,n}(t)\right\|_{C}
\notag \\
\geq &\frac{1}{2}\left|\Psi_{\beta,n}\left(t_{0}+\frac{\pi}{n}\right)-\Psi_{\beta,n}(t_{0})\right|
\notag \\
=&
\frac{1}{2}\Big|
g_{\psi,n}\Big(t_{0}+\frac{\pi}{n}\Big)\cos\Big(n\Big(t_{0}+\frac{\pi}{n}\Big)-\frac{\beta\pi}{2}\Big)
+h_{\psi,n}\Big(t_{0}+\frac{\pi}{n}\Big)\sin\Big(n\Big(t_{0}+\frac{\pi}{n}\Big)-\frac{\beta\pi}{2}\Big)\notag \\
-& \left(g_{\psi,n}(t_{0})\cos\Big(nt_{0}-\frac{\beta\pi}{2}\Big)
+h_{\psi,n}(t_{0})\sin\Big(nt_{0}-\frac{\beta\pi}{2}\Big)\right)\Big| \notag \\
=&
\frac{1}{2}\Big|-
g_{\psi,n}\Big(t_{0}+\frac{\pi}{n}\Big)-g_{\psi,n}(t_{0}) \Big|
=\frac{1}{2}\Big|
g_{\psi,n}\Big(\frac{\beta\pi-2\pi}{2n}\Big)+g_{\psi,n}\Big(\frac{\beta\pi}{2n}\Big) \Big|
\notag\\
=
&\frac{1}{2}\left| 2g_{\psi,n}(0)+
\left( g_{\psi,n}\Big(\frac{(\beta-2)\pi}{2n}\Big)-g_{\psi,n}(0)\right)
+\left(g_{\psi,n}\Big(\frac{\beta\pi}{2n}\Big)- g_{\psi,n}(0) \right) \right|
\notag\\
\geq
&\left| g_{\psi,n}(0)\right|-
\frac{1}{2}\left| g_{\psi,n}\Big(\frac{(\beta-2)\pi}{2n}\Big)-g_{\psi,n}(0)\right|
-\frac{1}{2}\left|g_{\psi,n}\Big(\frac{\beta\pi}{2n}\Big)- g_{\psi,n}(0) \right|,
\end{align}
де, як і раніше, $t_{0}=\frac{\beta\pi}{2n}$, $\beta\in[0,2]$.

За теоремою про  середнє значення
\begin{equation}\label{f_meanValue}
\left| g_{\psi,n}\Big(\frac{(\beta-2)\pi}{2n}\Big)-g_{\psi,n}(0)\right|
\leq
\| g_{\psi,n}^{'}\|_{C}\frac{|\beta-2|\pi}{2n}
\leq \frac{\pi}{n}\sum\limits_{k=1}^{\infty}k\psi(k+n).
\end{equation}

Аналогічно (див. (\ref{Psi_t0_2}))
\begin{equation}\label{f_meanValue1}
\left| g_{\psi,n}\Big(\frac{\beta\pi}{2n}\Big)-g_{\psi,n}(0)\right|
\leq \frac{\pi}{n}\sum\limits_{k=1}^{\infty}k\psi(k+n).
\end{equation}
Об'єднуючи (\ref{f7})-(\ref{f_meanValue1}), одержуємо шукану оцінку знизу для $I_{n}^{(3)}$
\begin{equation}\label{I3_estimate}
I_{n}^{(3)}
\geq \sum\limits_{k=n}^{\infty}\psi(k)-\frac{\pi}{n}\sum\limits_{k=1}^{\infty}k\psi(k+n).
\end{equation}
Лему~\ref{Lemma_norm} доведено.
\end{proof}

З формул (\ref{f3}), (\ref{f4}), (\ref{Norm2}) i (\ref{Norm_asympt}) випливає, що

\begin{equation*}
{\cal E}_{n}(C^{\psi}_{\beta,1})_{C}=
\frac{1}{\pi}
\sum\limits_{k=n}^{\infty}\psi(k)+
 \frac{\Theta_{2}}{n} \sum\limits_{k=1}^{\infty}k\psi(k+n).
 \end{equation*}
Теорему~\ref{theorem1} доведено. 
 \end{proof}

Зазначимо, що оцінки (\ref{Theorem2Ineq1}), (\ref{Theorem2Eq}) i (\ref{Theorem1Asymp}) є асимптотичними рівностями при $n\rightarrow\infty$, якщо виконується граничне співвідношення \eqref{LimitCase_BestPossibility}, тобто коли
\begin{equation}\label{AsympCondition}
\frac{1}{n}\sum\limits_{k=1}^{\infty}k\psi(k+n)=
o\left(\sum\limits_{k=n}^{\infty}\psi(k)
\right), \ \ n\rightarrow\infty.
\end{equation}
Умова (\ref{AsympCondition}), як буде показано нижче, має місце у ряді важливих випадків, зокрема, коли послідовність  $\psi(k)$ спадає до нуля при $k\rightarrow\infty$ швидше за довільну степеневу послідовність $\frac{1}{k^{r}}$, $r>1$.

\section{Наслідки з Теореми~\ref{theorem1} для класів аналітичних та цілих функцій }\label{corrolarySection_analyticFunctions}

Наведемо приклади важливих функціональних компактів $C^{\psi}_{\beta,1}$, для яких формула (\ref{Theorem1Asymp}) дозволяє записати асимптотичні  рівності для ${\cal E}_{n}(C^{\psi}_{\beta,1})_{C}$ при $n\rightarrow\infty$.

Розглянемо випадок, коли послідовності $\psi(k)$ задовольняють умову Даламбера $\mathcal{D}_{q}$, $q\in[0,1)$:
\begin{equation}\label{DalamberCondition}
\lim\limits_{k\rightarrow\infty}\frac{\psi(k+1)}{\psi(k)}=q, \ \ \ \psi(k)>0.
\end{equation}

Якщо $\psi(k)$ задовольняє умову \eqref{DalamberCondition} при деякому $q\in[0,1)$, то будемо записувати, що $\psi\in \mathcal{D}_{q}$. Нехай спочатку $q=0$.

Згідно з Теоремою 5 роботи \cite{Stepanets_Serdyuk_Shydlich}, твердження про існування послідовності $\psi\in \mathcal{D}_{0}$ такої, що для функції $f$ вірне
 включення $f\in C^{\psi}_{\beta}L_{1}$  
при будь-якому $\beta\in \mathbb{R}$, еквівалентне твердженню про  включення $f\in \mathcal{E}$, де $\mathcal{E}$ --- множина всіх $2\pi$--періодичних дійснозначних на дійсній осі функцій, які допускають аналітичне продовження на всю комплексну площину. Отже, класи $C^{\psi}_{\beta,1}$ при $\psi\in \mathcal{D}_{0}$ належать до множини $2\pi$--періодичних дійснозначних на $\mathbb{R}$ цілих функцій.

\begin{corollary}\label{cor1}
Нехай $\sum\limits_{k=n+1}^{\infty}k\psi(k)<\infty$, $\psi(k)\geq 0, \ k=1,2,...$, $n\in\mathbb{N}$ i $\beta\in \mathbb{R}$, тоді має місце рівномірна 
відносно  всіх параметрах оцінка  
\begin{equation}\label{f12}
{\cal E}_{n}(C^{\psi}_{\beta,1})_{C}=
\frac{1}{\pi}\psi(n)
+
\frac{\mathcal{O}(1)}{n}\sum\limits_{k=n+1}^{\infty}k\psi(k).
\end{equation}
Якщо, крім того, $\psi\in \mathcal{D}_{0}$, то оцінка \eqref{f12} є асимптотичною рівністю
при $n\rightarrow\infty$.
\end{corollary}

\begin{proof}
Користуючись формулою (\ref{Theorem1Asymp}) Теореми~\ref{theorem1},
 можна записати 
\begin{align*}
{\cal E}_{n}(C^{\psi}_{\beta,1})_{C}&=
\frac{1}{\pi}\psi(n)
+
\mathcal{O}(1)\left( \sum\limits_{k=1}^{\infty}\psi(k+n)
+
 \frac{1}{n} \sum\limits_{k=1}^{\infty}k\psi(k+n)
\right)
\notag \\
&=
\frac{1}{\pi}\psi(n)
+
\frac{\mathcal{O}(1)}{n} \sum\limits_{k=1}^{\infty}(k+n)\psi(k+n)
\notag \\
&=
\frac{1}{\pi}\psi(n)
+
\frac{\mathcal{O}(1)}{n} \sum\limits_{k=n+1}^{\infty}k\psi(k).
\end{align*}
Тим самим оцінку \eqref{f12} доведено. Покажемо, що при $\psi\in \mathcal{D}_{0}$
\begin{equation}\label{f12_111}
\frac{1}{n}\sum\limits_{k=n+1}^{\infty}k\psi(k)=o\left(\psi(n)\right), n\rightarrow\infty.
\end{equation}
Виберемо номери $n$ такими, щоб
\begin{equation}\label{f12_113}
\frac{\psi(k+1)}{\psi(k)}<\frac{1}{2}, \ \ \ k=n, \ n+1, ...
\end{equation}

Тоді, з урахуванням \eqref{f12_113}, маємо
\begin{align}\label{f12_114}
&\frac{1}{n} \sum\limits_{k=1}^{\infty}k\psi(k) = \left(1+\frac{1}{n} \right)\psi(n+1)
+\frac{\psi(n+1)}{n}\sum\limits_{j=2}^{\infty}(n+j)\prod\limits_{\ell=1}^{j-1}\frac{\psi(n+\ell+1)}{\psi(n+\ell)}
\notag \\
&<\psi(n+1) \left(2+\frac{1}{n}\sum\limits_{j=2}^{\infty}\frac{2j}{2^{j-1}} \right)
<\psi(n+1) \left(2+\frac{4}{n}\sum\limits_{j=1}^{\infty}\frac{j}{2^{j}} \right)
<10 \psi(n+1).
\end{align}

Оскільки, в силу  $\psi\in \mathcal{D}_{0}$
\begin{equation}\label{f12_115}
\psi(n+1)= o\left( \psi(n)\right), \ \ n\rightarrow\infty,
\end{equation}
то із \eqref{f12_114} i \eqref{f12_115} випливає \eqref{f12_111}.

Наслідок~\ref{cor1} доведено.
\end{proof}

Зауважимо, що асимптотичну рівність  (\ref{f12}) з залишковим членом, записаним в іншій формі,  було отримано раніше в   \cite{Serdyuk2005} і \cite{Serdyuk2005Lp}. При $\psi\in \mathcal{D}_{0}$ оцінки залишкового члена в   \cite{Serdyuk2005} і \cite{Serdyuk2005Lp} є більш точними, ніж у формулі \eqref{f12}.

Типовими представниками послідовностей, що задовольняють умову $\mathcal{D}_{0}$ є послідовності $\psi(k)=e^{-\alpha k^{-r}}$,  $r>1$, $\alpha>0$. Для породжуваних такими послідовностями класів 
$C^{\psi}_{\beta,1}=C^{\alpha,r}_{\beta,1}$, одержуємо наступне твердження.

\begin{corollary}\label{cor01}
Нехай  $r>1$, $\alpha>0$ і  $\beta\in\mathbb{R}$. Тоді, при $n\geq \left(\frac{3}{\alpha r} \right)^{\frac{1}{r}}-1$, $n\in\mathbb{N}$, має місце рівномірна по всіх розглядуваних параметрах оцінка
\begin{equation}\label{f41}
{\cal E}_{n}(C^{\alpha,r}_{\beta,1})_{C}= 
e^{-\alpha n^{r}}\Big(
\frac{1}{\pi}+\mathcal{O}(1)e^{-\alpha r n^{r-1}}\Big(1+\frac{1}{\alpha r (n+1)^{r-2}}\Big)\Big).
\end{equation}
\end{corollary}

\begin{proof}
З формули (\ref{f12}) випливає, що 
\begin{equation}\label{f42}
{\cal E}_{n}(C^{\alpha,r}_{\beta,1})_{C}= 
\frac{1}{\pi}e^{-\alpha n^{r}}
+
\frac{\mathcal{O}(1)}{n}\sum\limits_{k=n+1}^{\infty}ke^{-\alpha k^{r}}.
\end{equation}

Легко переконатись, що при номерах $n$ таких, що $(n+1)^{r}> \frac{1}{\alpha r}$
\begin{equation}\label{f43}
\frac{1}{n}\sum\limits_{k=n+1}^{\infty}ke^{-\alpha k^{r}}
<\frac{1}{n}\left((n+1)e^{-\alpha (n+1)^{r}}+
\int\limits_{n+1}^{\infty}te^{-\alpha t^{r}}dt \right).
\end{equation}

Інтегруючи частинами, отримуємо
\begin{align}\label{f44}
\int\limits_{n+1}^{\infty}te^{-\alpha t^{r}}dt &=
\int\limits_{n+1}^{\infty}t^{2}  \frac{1}{\alpha r t^{r}} \left(-e^{-\alpha t^{r}}\right)' dt
\leq
 \frac{1}{\alpha r (n+1)^{r}} \int\limits_{n+1}^{\infty}t^{2} \left(-e^{-\alpha t^{r}}\right)' dt
 \notag \\
 &= \frac{1}{\alpha r (n+1)^{r}}  \left((n+1)^{2}e^{-\alpha (n+1)^{r}}+2\int\limits_{n+1}^{\infty}te^{-\alpha t^{r}}dt \right).
\end{align}
З останньої нерівності маємо
\begin{equation}\label{f45}
\left(1- \frac{2}{\alpha r (n+1)^{r}} \right)  \int\limits_{n+1}^{\infty}te^{-\alpha t^{r}}dt 
\leq
\frac{(n+1)^{2}e^{-\alpha (n+1)^{r}}}{\alpha r (n+1)^{r}},
\end{equation}
що рівносильно тому, що 
\begin{align}\label{f46}
  \int\limits_{n+1}^{\infty}te^{-\alpha t^{r}}dt 
&\leq
\frac{e^{-\alpha (n+1)^{r}}}{\alpha r (n+1)^{r-2}}\frac{\alpha r(n+1)^{r}}{\alpha r(n+1)^{r}-2} \notag \\
&=
\frac{e^{-\alpha (n+1)^{r}}}{\alpha r (n+1)^{r-2}}\left(1+ \frac{2}{\alpha r(n+1)^{r}-2}\right).
\end{align}
Зі співвідношень (\ref{f43}) і (\ref{f46}) випливає, що
\begin{equation}\label{f47}
\frac{1}{n}\sum\limits_{k=n+1}^{\infty}ke^{-\alpha k^{r}}
=
\mathcal{O}(1)  \left( e^{-\alpha (n+1)^{r}}+
\frac{e^{-\alpha (n+1)^{r}}}{\alpha r (n+1)^{r-2}}\left(1+ \frac{2}{\alpha r(n+1)^{r}-2}\right)
\right).
\end{equation}

Об'єднавши (\ref{f42}) і (\ref{f47}), одержуємо, що при всіх номерах $n$ таких, що $(n+1)^{r}> \frac{3}{\alpha r}$
\begin{align*}
{\cal E}_{n}(C^{\alpha,r}_{\beta,1})_{C}=& 
\frac{1}{\pi}e^{-\alpha n^{r}}
+
\mathcal{O}(1)\left( e^{-\alpha (n+1)^{r}}+
\frac{e^{-\alpha (n+1)^{r}}}{\alpha r (n+1)^{r-2}}\left(1+ \frac{2}{\alpha r(n+1)^{r}-2}\right)
\right) \notag \\
=& 
e^{-\alpha n^{r}}\left(\frac{1}{\pi}
+
\mathcal{O}(1)\left( e^{-\alpha r n^{r-1}}+
\frac{e^{-\alpha r n^{r-1}}}{\alpha r (n+1)^{r-2}}
\right)\right) .
\end{align*}
Наслідок~\ref{cor01} доведено.
\end{proof}

Формулу (\ref{f41}) із залишковим членом, записаним дещо в іншому вигляді було отримано в \cite{Serdyuk2005} і \cite{Serdyuk2005Lp}. При цьому оцінки з  \cite{Serdyuk2005} і \cite{Serdyuk2005Lp} містять більш точні оцінки залишкового члена ніж у (\ref{f41}).



Нехай далі $q\in (0,1)$. Згідно з Теоремою 3 роботи \cite{Stepanets_Serdyuk_Shydlich}, твердження про існування послідовності  $\psi\in \mathcal{D}_{q}$, $q\in (0,1)$ такої, що для функції $f$  вірне  включення $C^{\alpha,1}_{\beta}L_{1}$ при будь-якому $\beta\in \mathbb{R}$ еквівалентне твердженню про включення $f\in \mathcal{A}$, де $\mathcal{A}$ --- множина всіх $2\pi$--періодичних дійснозначних на дійсній осі функцій, які допускають аналітичне продовження на деяку смугу $\left| \mathrm{Im} z \right| <c$, $c>0$, комплексної площини.
Отже класи $C^{\psi}_{\beta,1}$, $\psi\in \mathcal{D}_{q}$, $0<q<1$, складаються з періодичних, аналітичних у смузі $\left| \mathrm{Im} \, z \right| <c$ функцій, при цьому $c= \ln \frac{1}{q}$ (див., наприклад, \cite[c. 32]{Step monog 1987}).

Послідовності $\psi(k)=e^{-\alpha k}$, $\alpha>0$ належать до множини $\mathcal{D}_{q}$ при $q=e^{-\alpha}$, а відповідні класи $C^{\psi}_{\beta,1}=C^{\alpha,1}_{\beta,1}$ породжуються ядрами Пуассона
\begin{equation}\label{kernelPsiDq}
P_{\alpha,1,\beta}(t)=\sum\limits_{k=1}^{\infty}e^{-\alpha k}\cos
\big(kt-\frac{\beta\pi}{2}\big), \ \alpha>0, \  \beta\in
    \mathbb{R}.
\end{equation}
 
Із Теореми~\ref{theorem1} для класів $C^{\alpha,1}_{\beta,1}$  отримуємо наступне твердження.


\begin{corollary}\label{cor2}
Нехай $\alpha>0$ і $\beta\in \mathbb{R}$.
 Тоді, при всіх  $n\in\mathbb{N}$ має місце рівність
\begin{equation}\label{f31}
{\cal E}_{n}(C^{\alpha,1}_{\beta,1})_{C}=e^{-\alpha n}\left(
\frac{1}{\pi}\frac{1}{1-e^{-\alpha}}
+
\frac{\Theta}{n}\frac{e^{-\alpha}}{(1-e^{-\alpha})^{2}}\right),
\end{equation}
де  для величини $\Theta=\Theta(n,\alpha,\beta)$ виконуються нерівності $-1\leq \Theta\leq 0$.
\end{corollary}

\begin{proof}
Покладемо $q=e^{-\alpha}$. Тоді, 
з Теореми~\ref{theorem1} випливає, що при всіх $n\in\mathbb{N}$
\begin{align}\label{f32}
{\cal E}_{n}(C^{\alpha,1}_{\beta,1})_{C}=&
\frac{1}{\pi}
\sum\limits_{k=n}^{\infty}q^{k}+
 \frac{\Theta}{n} \sum\limits_{k=0}^{\infty}kq^{k+n} \notag \\
 =&\frac{1}{\pi}\frac{q^{n}}{1-q}
+
\frac{\Theta}{n}\left( \sum\limits_{k=n}^{\infty}kq^{k}-n\sum\limits_{k=n}^{\infty}q^{k}  \right) \notag \\
=&
\frac{1}{\pi}\frac{q^{n}}{1-q}
+
\frac{\Theta}{n}\left( \frac{nq^{n}(1-q)+q^{n+1}}{(1-q)^{2}}-\frac{nq^{n}}{1-q} \right) \notag \\
=&
\frac{1}{\pi}\frac{q^{n}}{1-q}
+
\frac{\Theta}{n}\frac{q^{n+1}}{(1-q)^{2}},
 \end{align}
 де була  використана наступна рівність:
 \begin{equation*}
 \sum\limits_{k=n}^{\infty}kq^{k}=\frac{nq^{n}(1-q)+q^{n+1}}{(1-q)^{2}}, \ \ q\in (0,1), \ \ n\in\mathbb{N}.
 \end{equation*}
 Наслідок~\ref{cor2} доведено.
\end{proof}

Оцінка (\ref{f31}) уточнює асимптотичні рівності для величин ${\cal E}_{n}(C^{\alpha,r}_{\beta,1})_{C}$, які були встановлені в  \cite{Serdyuk2005} і \cite{Serdyuk2005Lp}. Асимптотичні рівності для величин  ${\cal E}_{n}(C^{\psi}_{\beta,1})_{C}$ при $\psi\in \mathcal{D}_{q}$, $q\in(0,1)$, містяться у наступному твердженні.

\begin{corollary}\label{corDq}
Нехай $\psi\in \mathcal{D}_{q}$, $q\in(0,1)$,  $\beta\in\mathbb{R}$, $n\in \mathbb{N}$. Тоді при  всіх номерах $n$ таких, що
\begin{equation}\label{Number_n_Dq}
\frac{1}{n}+\varepsilon_{n}<\frac{1-q}{2},
\end{equation}
де
\begin{equation}\label{Varepselon_n_Dq}
\varepsilon_{n}:=\sup\limits_{k\geq n} \left| \frac{\psi(k+1)}{\psi(k)}-q\right|,
\end{equation}
має місце рівномірна відносно всіх розглядуваних параметрів оцінка
\begin{equation}\label{Estimate_corDq}
{\cal E}_{n}(C^{\psi}_{\beta,1})_{C}= 
\psi(n) \left(
\frac{1}{\pi(1-q)}+\mathcal{O}(1)\Big(\frac{q}{n(1-q)^{2}}+\frac{\varepsilon_{n}}{(1-q)^{2}}\Big)\right).
\end{equation}
\end{corollary}

\begin{proof} 
З Леми 1 роботи \cite{StepanetsSerdyuk2000No3}  випливає, що при $\psi\in \mathcal{D}_{q}$, $0<q<1$, $n\in \mathbb{N}$, має місце рівність 
\begin{equation}\label{Sum_psi_Dq}
\sum\limits_{k=n}^{\infty}\psi(k)=
\psi(n) \left(\frac{1}{q^{n}}\sum\limits_{k=n}^{\infty}q^{k} +r_{n}\right),
\end{equation}
де для залишку $r_{n}$ при всіх номерах $n$ таких, що
\begin{equation}\label{Proof_Number_n_Dq}
\varepsilon_{n}<\frac{1-q}{2}
\end{equation}
виконується оцінка
\begin{equation}\label{Estimate_r_n_corDq}
\left| r_{n}\right|
\leq \frac{\varepsilon_{n}}{(1-q-\varepsilon_{n}) (1-q)}\leq
\frac{2\varepsilon_{n}}{ (1-q)^{2}} .
\end{equation}
Очевидно, що якщо $\psi\in \mathcal{D}_{q}$, $0<q<1$, то і послідовність $k\psi(k)$ також задовольняє умову $\mathcal{D}_{q}$, а тому знову ж таки в силу Леми~1 із \cite{StepanetsSerdyuk2000No3} 
\begin{equation}\label{Sum_kpsi_Dq}
\sum\limits_{k=n+1}^{\infty} k\psi(k)=
(n+1)\psi(n+1) \left(\frac{1}{q^{n+1}}\sum\limits_{k=n+1}^{\infty}q^{k} +r_{n+1}^{*}\right),
\end{equation}
 де для залишку $r_{n+1}^{*}$ при усіх номерах $n$ таких, що
 \begin{equation}\label{Varepselon_Star_n_Dq}
\varepsilon_{n+1}^{*}:=\sup\limits_{k\geq n+1} \left| \frac{\psi(k+1)(k+1)}{\psi(k)k}-q\right|< \frac{1-q}{2}
\end{equation}
виконується оцінка 
\begin{equation}\label{Estimate_r_n_Star_corDq}
\left| r_{n+1}^{*}\right|
\leq
\frac{2\varepsilon_{n+1}^{*}}{ (1-q)^{2}} .
\end{equation}
Із означень величин $\varepsilon_{n}$ i $\varepsilon_{n+1}^{*}$ (див. \eqref{Varepselon_n_Dq} i \eqref{Varepselon_Star_n_Dq}) маємо
\begin{equation}\label{Estimate_Varepselon_Star_n_Dq}
\varepsilon_{n+1}^{*} \leq \sup\limits_{k\geq n+1} \left| \frac{\psi(k+1)}{\psi(k)}-q\right|+\frac{1}{n+1}= \varepsilon_{n+1}+\frac{1}{n+1}< \varepsilon_{n}+\frac{1}{n}.
\end{equation}
Із \eqref{Estimate_Varepselon_Star_n_Dq} видно, що виконання нерівності \eqref{Number_n_Dq} гарантує і виконання нерівностей \eqref{Proof_Number_n_Dq} i \eqref{Varepselon_Star_n_Dq}, а отже і оцінок \eqref{Estimate_r_n_corDq} i \eqref{Estimate_r_n_Star_corDq} для залишків у рівностях \eqref{Sum_psi_Dq} i \eqref{Sum_kpsi_Dq}.

Тоді в силу оцінки \eqref{Theorem1Asymp} Теореми~\ref{theorem1} і рівностей  \eqref{Sum_psi_Dq} i \eqref{Sum_kpsi_Dq} випливає, що при всіх номерах $n$, які задовольняють умову \eqref{Number_n_Dq}, справджуються співвідношення
\begin{align}\label{Formulas_Corrolary}
{\cal E}_{n}(C^{\psi}_{\beta,1})_{C}= &
\frac{1}{\pi}\sum\limits_{k=n}^{\infty}\psi(k) +
\frac{\mathcal{O}(1)}{n}\sum\limits_{k=1}^{\infty} k\psi(k+n) \notag \\
=&
\frac{\psi(n)}{\pi}\left( \frac{1}{q^{n}} \sum\limits_{k=n}^{\infty} q^{k} +r_{n}\right)+
 \frac{\mathcal{O}(1)}{n}\sum\limits_{k=n+1}^{\infty} (k-n)\psi(k) \notag \\
 =&
\frac{\psi(n)}{\pi}
\left( \frac{1}{1-q} +\mathcal{O}(1)\frac{\varepsilon_{n}}{(1-q)^{2}}
\right)+
\mathcal{O}(1)\left(  \frac{1}{n}\sum\limits_{k=n+1}^{\infty} k\psi(k)
-\sum\limits_{k=n+1}^{\infty} \psi(k)
\right) \notag \\
=& 
\psi(n)
\left( \frac{1}{\pi(1-q)} +\mathcal{O}(1)\frac{\varepsilon_{n}}{(1-q)^{2}}
\right) \notag \\
+ & \mathcal{O}(1) \psi(n+1) \left( \frac{n+1}{n}  \left(\frac{1}{q^{n+1}}\sum\limits_{k=n+1}^{\infty}q^{k}+ \frac{\varepsilon_{n+1}^{*}}{(1-q)^{2}} \right)
-\frac{1}{q^{n+1}} \sum\limits_{k=n+1}^{\infty}q^{k}  + \frac{\varepsilon_{n+1}}{(1-q)^{2}} \right)\notag \\
=& 
\psi(n)
\left( \frac{1}{\pi(1-q)} +\mathcal{O}(1)\frac{\varepsilon_{n}}{(1-q)^{2}}
\right)+
 \mathcal{O}(1) \psi(n+1) \left( \frac{1}{n(1-q)}+ \frac{\varepsilon_{n}+\frac{1}{n}}{(1-q)^{2}}\right)
  \notag \\
  =&
  \psi(n)
\left( \frac{1}{\pi(1-q)} +\mathcal{O}(1)
\left( \frac{\varepsilon_{n}}{(1-q)^{2}} +\frac{\psi(n+1)}{\psi(n)}\frac{1}{n(1-q)^{2}}  \right)
\right)  \notag \\
  =&
  \psi(n)
\left( \frac{1}{\pi(1-q)} +\mathcal{O}(1)
\left( \frac{\varepsilon_{n}}{(1-q)^{2}} + \frac{q}{n(1-q)^{2}}  \right)
\right).
\end{align}
Наслідок \ref{corDq} доведено.
\end{proof}

Асимптотичні рівності \eqref{Estimate_corDq} вперше  були встановлені в роботах \cite{Serdyuk2005} i \cite{Serdyuk2005Lp}.

\section{Наслідки з Теореми~\ref{theorem1} для класів нескінченно диференційовних  функцій }\label{Section_corrolary_infinitelyDifferentiable}

В даному підрозділі будемо вважати, що послідовності $\psi(k)$, що породжують множини $C^{\psi}_{\beta}L_{1}$ та $C^{\psi}_{\beta,1}$, є звуженням на множину натуральних чисел деяких додатних неперервних опуклих донизу функцій  $\psi(t)$ неперервного аргументу $t\geq 1$, що прямують до нуля при $t\rightarrow\infty$.
Множину всіх таких функцій $\psi$ позначають через $\mathfrak{M}$:
\begin{equation}\label{Mathfrak_M}
\mathfrak{M}\!=\! \left\{\psi\!\in \! C[1,\infty)\!: \psi(t)\!>\!0, \psi(t_{1}-2\psi((t_{1}+t_{2})/2) +\psi(t_{2})\geq 0 \ \forall t_{1}, t_{2} \in[1,\infty), \  \lim\limits_{t\rightarrow\infty}\psi(t)\!=\!0 \right\}.
\end{equation}

Наслідуючи О.І. Степанця (див., наприклад, \cite[с. 160]{Stepanets1}), кожній функції $\psi\in\mathfrak{M}$ поставимо у відповідність характеристики
\begin{equation*}
\eta(t)=\eta(\psi;t)=\psi^{-1}\left( \frac{1}{2}\psi(t)\right) 
\end{equation*}
та
\begin{equation*}
\mu(t)=\mu(\psi;t)=\frac{t}{\eta(t)-t},
\end{equation*}
де $\psi^{-1}$ --- обернена до $\psi$ функція, і покладемо
\begin{equation*}
\mathfrak{M}_{\infty}^{+}= \left\{\psi\in \mathfrak{M}: \ \mu(t)\uparrow, \ \ t\rightarrow\infty \right\}.
\end{equation*}

Через $\mathfrak{M}^{\alpha}$ позначимо підмножину всіх функцій 
$\psi\in \mathfrak{M}$, для яких величина
 \begin{equation}\label{psi_alpha}
\alpha(t)=\alpha(\psi;t):=\frac{\psi(t)}{t|\psi'(t)|}, \ \ \psi'(t):=\psi'(t+0),
\end{equation}
спадає до нуля при $t\rightarrow\infty$:
\begin{equation}\label{Mathfrak_M_alpha}
\mathfrak{M}^{\alpha}= \left\{\psi\in \mathfrak{M}:   \ \ \lim\limits_{t\rightarrow\infty}\alpha(\psi;t)=0 \right\}.
\end{equation}

Згідно з Теоремою 2 роботи \cite{Stepanets_Serdyuk_Shydlich2007},  твердження про існування послідовності  $\psi\in \mathfrak{M}^{\alpha}$ (або $\psi\in\mathfrak{M}_{\infty}^{+}$), такої, що для функції $f$ вірне  включення $f\in C^{\psi}_{\beta}L_{1}$  при будь-якому $\beta\in\mathbb{R}$, еквівалентне твердженню про включення $f\in D^{\infty}$, де $D^{\infty}$ --- множина всіх нескінченно диференційовних $2\pi$-періодичних дійснозначних функцій.
А отже, класи $C^{\psi}_{\beta,1}$ при $\psi\in \mathfrak{M}^{\alpha}$ (або $\psi\in\mathfrak{M}_{\infty}^{+}$),  є класами нескінченно диференційовних періодичних функцій.  В тій же роботі було показано, що має місце включення
\begin{equation}\label{Mathfrak_M_alpha}
\mathfrak{M}_{\infty}^{+} \subset \mathfrak{M}^{\alpha}\subset \mathfrak{M}^{\infty}= \left\{\psi\in \mathfrak{M}:  \ \forall r>0 \ \ \lim\limits_{t\rightarrow\infty} t^{r}\psi(t)=0 \right\},
\end{equation}
яке означає, що функції $\psi(t)$ із $\mathfrak{M}^{\alpha}$ спадають до нуля швидше за довільну степеневу функцію.

Для величин  ${\cal E}_{n}(C^{\psi}_{\beta,1})_{C}$,  $\psi\in\mathfrak{M}_{\infty}^{+}$ за умови $\eta(n)-n>2$ відомі точні порядкові рівності
\begin{equation}\label{OrderEstimatesUMJ2014}
{\cal E}_{n}(C^{\psi}_{\beta,1})_{C}
\asymp
\psi(n) (\eta(n)-n),
 \end{equation}
які співпрадають з точними порядковими рівностями для найкращих рівномірних наближень тригонометричними поліномами порядку $n-1$
\begin{equation*}
{ E}_{n}(C^{\psi}_{\beta,1})_{C}
=\inf\limits_{t_{n-1}\in \mathcal{T}} \| f-t_{t_{n-1}}  \|_{C}
 \end{equation*}
а саме, (див., наприклад,  \cite{Serdyuk_Stepaniuk2014})
\begin{equation}\label{OrderEstimatesUMJ2014_2}
{ E}_{n}(C^{\psi}_{\beta,1})_{C}
\asymp
{\cal E}_{n}(C^{\psi}_{\beta,1})_{C}
\asymp
\psi(n) (\eta(n)-n),
 \end{equation}
 (тут і надалі запис $A(n)\asymp B(n)$ для додатних послідовностей $A(n)$ i $B(n)$ означає існування додатних констант $K_{1}$ i $K_{2}$ таких, що 
 $K_{1}B(n) \leq A(n) \leq K_{2} B(n)$, $n\in \mathbb{N}$). 
 
 Як показано в \cite[с. 166]{Stepanets2} для довільної функції $\psi$ із $\mathfrak{M}_{\infty}^{+}$ має місце порядкова рівінсть
\begin{equation}\label{OrderEstimates_theta_and_lambda}
\eta(t)-t \asymp \lambda(t), \ t\geq 1,
 \end{equation}
де $\lambda(t)$ --- характеристика вигляду
\begin{equation}\label{psi_lambda}
\lambda(t)=\lambda(\psi;t):=\frac{\psi(t)}{|\psi'(t)|}.
\end{equation}

З урахуванням \eqref{OrderEstimatesUMJ2014} можна записати у вигляді 
\begin{equation}\label{OrderEstimatesUMJ2014_3}
{\cal E}_{n}(C^{\psi}_{\beta,1})_{C}
\asymp
\psi(n) \lambda(n).
 \end{equation}
 
 Наступне твердження містить сильну асимптотику  величин ${\cal E}_{n}(C^{\psi}_{\beta,1})_{C}$ , $\psi\in\mathfrak{M}^{\alpha}$ при деяких природних обмеженнях на $\alpha(t)$ i $\lambda(t)$.

\begin{theorem}\label{theorem2_M}
Нехай $\beta\in\mathbb{R}$, $\psi\in\mathfrak{M}$ i характеристики \eqref{psi_alpha} i  \eqref{psi_lambda} задовольняють умови
\begin{equation}\label{alphaTo_0}
\alpha(t)\downarrow0,
\end{equation}
\begin{equation}\label{lambdaTo_infty}
\lambda(t)\uparrow \infty, \ \ t\rightarrow\infty.
\end{equation}

Тоді для всіх $n\in \mathbb{N}$ таких, що
\begin{equation}\label{alpha_1/4}
\alpha(n)\leq \frac{1}{4}
\end{equation}
виконується оцінка
\begin{equation}\label{Theorem2Asymp}
{\cal E}_{n}(C^{\psi}_{\beta,1})_{C}=
\psi(n) \lambda(n)
\left( \frac{1}{\pi}
+
\frac{\xi_{1}}{\lambda(n)}+\xi_{2}\alpha(n) 
 \right),
 \end{equation}
 де $-1\leq \xi_{1}\leq 1+\frac{1}{\pi}$ та $\, \, -4\leq \xi_{2}\leq \frac{4}{3}\left(1+\frac{1}{\pi}\right)$.
\end{theorem}

\begin{proof}[Доведення Теореми~\ref{theorem2_M}] 
Для оцінки величини ${\cal E}_{n}(C^{\psi}_{\beta,1})_{C}$  використаємо формулу \eqref{Theorem1Asymp} із Теореми~\ref{theorem1}. При цьому нам буде необхідно знайти оцінки рядів  $\Sigma_{1}=\sum\limits_{k=n}^{\infty} \psi(k)$ та $\Sigma_{2}=\sum\limits_{k=n}^{\infty} k \psi(k)$.

В силу монотонного спадання функції $\psi\in\mathfrak{M}$ бачимо, що
\begin{equation}\label{int_psi_monot}
\int\limits_{n}^{\infty} \psi(t)dt \leq \sum\limits_{k=n}^{\infty}\psi(k) \leq
\psi(n)+ \int\limits_{n}^{\infty} \psi(t)dt,
\end{equation}
а, отже,
\begin{equation}\label{Sum_psi_Int}
 \sum\limits_{k=n}^{\infty}\psi(k)=  \int\limits_{n}^{\infty} \psi(t)dt + \Theta_{4} \psi(n), \ \ 0\leq \Theta_{4}\leq 1.
\end{equation}

Оцінка інтеграла  $\int\limits_{n}^{\infty}\psi(t) dt$ випливає з наступного твердження, яке може мати і самостійний інтерес.

\begin{lemma}\label{Lemma_Estimate_Integral_psi}
Нехай $\psi\in \mathfrak{M}$, $\lambda(t)$ монотонно неспадає, а $\alpha(t)$ монотонно незростає на $[1,\infty)$. Тоді при всіх $a\geq 1$, таких, що $\alpha(a)<1$, виконуються оцінки
\begin{equation}\label{Lemma_Estimate_Integral_psi_Inequality}
\lambda(a) \psi(a)\leq  \int\limits_{a}^{\infty}\psi(t) dt \leq \lambda(a) \psi(a) \left(1+\frac{\alpha(a)}{1-\alpha(a)}\right).
\end{equation}
\end{lemma}
\begin{proof}[Доведення Леми~\ref{Lemma_Estimate_Integral_psi}] Оскільки в силу включення $\psi\in \mathfrak{M}$, функція $\psi(t)$ є локально абсолютно неперервною на $[1,\infty)$, то, враховуючи монотонне неспадання $\lambda(t)$, одержуємо шукану оцінку знизу
\begin{align}\label{Int_psi_EstimateBelow}
I_{1}:= \int\limits_{a}^{\infty}\psi(t) dt=
 \int\limits_{a}^{\infty}-\psi'(t)\lambda(t) dt\geq
 \lambda(a)\int\limits_{a}^{\infty} (-\psi'(t)) dt =
  \psi(a)\lambda(a),
\end{align}

З іншого боку, враховуючи монотонне незростання функції $\alpha(t)$, і застосовуючи метод інтегрування частинами, маємо
\begin{align}\label{Int_psi_Estimate}
I_{1}=& \int\limits_{a}^{\infty}\psi(t) dt=
 \int\limits_{a}^{\infty} \alpha(t)(-\psi'(t)t) dt\leq
\alpha(a) \int\limits_{a}^{\infty} (-\psi'(t)t)  \notag \\
= &
\alpha(a)\left(
 \psi(a)a+  \int\limits_{a}^{\infty}\psi(t) dt 
\right)=
  \psi(a)\lambda(a)+\alpha(a) I_{1}.
\end{align}

З \eqref{Int_psi_Estimate} одержуємо, що
\begin{equation}\label{Int_psi_EstimateUpper}
I_{1}\leq
 \frac{\lambda(a) \psi(a)}{1-\alpha(a)}=
 \lambda(a)\psi(a) \left(1+\frac{\alpha(a)}{1-\alpha(a)} \right).
 \end{equation}
Із \eqref{Int_psi_EstimateBelow} і \eqref{Int_psi_EstimateUpper} випливає \eqref{Lemma_Estimate_Integral_psi_Inequality}. Лему~\ref{Lemma_Estimate_Integral_psi} доведено.

\end{proof}

Застосування Леми~\ref{Lemma_Estimate_Integral_psi} при $a=n$, $n\in\mathbb{N}$, за умови \eqref{alpha_1/4} дозволяє записати, що
\begin{align}\label{Int_psi_EstimateAsymp}
I_{1}= \int\limits_{a}^{\infty}\psi(t) dt=
 \psi(n)\lambda(n) \left( 1+ \Theta_{5}\alpha(n) \right), \ \ 0\leq \Theta_{5} \leq \frac{4}{3}.
\end{align}

Отже, з урахуванням \eqref{Sum_psi_Int} i \eqref{Int_psi_EstimateAsymp} при $\alpha(n) \leq \frac{1}{4}$
\begin{equation}\label{Sum_psi(k)}
 \sum\limits_{k=n}^{\infty}\psi(k)= \psi(n)\lambda(n) \left( 1+ \frac{\Theta_{4}}{\lambda(n)}+ \Theta_{5}\alpha(n) \right), \ \ 0\leq \Theta_{5} \leq \frac{4}{3}, \ 0\leq \Theta_{4}\leq 1.
\end{equation}

Далі знайдемо оцінку для $\Sigma_{2}=\sum\limits_{k=n}^{\infty}k\psi(k)$.
В силу \eqref{alpha_1/4} функція $t\psi(t)$ спадає на $[n,\infty)$, а тому
\begin{equation}\label{Sum_k_psi_Formula}
\int\limits_{n}^{\infty} t\psi(t) dt \leq \sum\limits_{k=n}^{\infty}k\psi(k) \leq n\psi(n)+ \int\limits_{n}^{\infty}t\psi(t) dt,
\end{equation}
і, отже,
\begin{equation}\label{Sum_k_psi_Formula1}
\sum\limits_{k=n}^{\infty}k\psi(k)=\int\limits_{n}^{\infty} t\psi(t) dt + 
\Theta_{6} n\psi(n), \ \ 0\leq \Theta_{6}\leq 1.
\end{equation}

Для оцінки інтеграла $I_{2}=\int\limits_{n}^{\infty}t\psi(t) dt$ знову використаємо метод інтегрування частинами і врахуємо  \eqref{Mathfrak_M_alpha} та умову  незростання $\alpha(n)$

 \begin{align}\label{Int2_Formula}
I_{2}=& \int\limits_{n}^{\infty}t\psi(t) dt=
  \int\limits_{n}^{\infty}t^{2}  \frac{\psi(t)}{-t\psi'(t)} (-\psi'(t))dt\leq
 \alpha(n) \int\limits_{n}^{\infty}t^{2}(-\psi'(t))dt
 \notag \\
=&\alpha(n) \left(n^{2}\psi(n)+2I_{2} \right),
\end{align}

З останніх співвідношень і умови \eqref{alpha_1/4} маємо
  \begin{equation*}
I_{2}\left(1-2\alpha(n) \right)\leq \alpha(n)n^{2}\psi(n)
 \end{equation*}
 і, отже, з урахуванням умови  \eqref{alpha_1/4} маємо
\begin{align}\label{Int2_Formula_estimate}
I_{2}\leq& \psi(n) n^{2} \alpha(n)\frac{1}{1-2\alpha(n)}
=\psi(n) n^{2} \alpha(n)\left(1+ \frac{2\alpha(n)}{1-2\alpha(n)}\right)
 \notag \\
\leq & \psi(n) n^{2} \alpha(n)\left(1+4\alpha(n) \right)
=\psi(n) n \lambda(n) (1+4\alpha(n)).
\end{align}

З іншого боку, з урахуванням умови \eqref{Int_psi_EstimateAsymp},
\begin{align}\label{Int2_FormulaBelow}
I_{2}= \int\limits_{n}^{\infty}t\psi(t) dt \geq
 n \int\limits_{n}^{\infty}\psi(t) dt 
  \geq \psi(n) n \lambda(n).
\end{align}

Об'єднання \eqref{Int2_Formula_estimate} i \eqref{Int2_FormulaBelow} дозволяє записати, що при  $\alpha(n) \leq \frac{1}{4}$
\begin{align}\label{Int2_t_psi}
 \int\limits_{n}^{\infty}t\psi(t) dt= \psi(n) n \lambda(n) \left(1+\Theta_{7}\alpha(n) \right),
 \ \ \ 0\leq \Theta_{7}\leq 4.
\end{align}

Із формул \eqref{Sum_k_psi_Formula1} i \eqref{Int2_t_psi} випливає, що за умов \eqref{psi_lambda} i \eqref{alpha_1/4}
\begin{align}\label{Sum_k_psi}
\sum\limits_{k=n}^{\infty} k\psi(k)= \psi(n) n \lambda(n) \left(1+\Theta_{7}\alpha(n)+
\frac{\Theta_{6}}{\lambda(n)} \right),
 \ \ \ 0\leq \Theta_{7}\leq 4, \ \ 0\leq \Theta_{6}\leq 1.
\end{align}

Користуючись оцінками \eqref{Sum_k_psi_Formula} i \eqref{Sum_psi(k)}, одержуємо
\begin{align}\label{form6}
&\frac{1}{n}\sum\limits_{k=1}^{\infty}k\psi(k+n)
=\frac{1}{n}\sum\limits_{k=0}^{\infty}k\psi(k+n)
\notag\\
=&\frac{1}{n} \left(\sum\limits_{k=n}^{\infty}k\psi(k)
-n\sum\limits_{k=n}^{\infty}\psi(k)
\right) \notag\\
=&\frac{1}{n} \sum\limits_{k=n}^{\infty}k\psi(k)
-\sum\limits_{k=n}^{\infty}\psi(k) \notag \\
=
&\psi(n) \lambda(n) \left(1+\Theta_{7}\alpha(n)+
\frac{\Theta_{6}}{\lambda(n)} \right)
-
 \psi(n) n \lambda(n) \left(1+\Theta_{5}\alpha(n)+
\frac{\Theta_{4}}{\lambda(n)} \right)\notag \\
 =& 
 \psi(n)  \lambda(n) \left((\Theta_{7}-\Theta_{5})\alpha(n)+
\frac{\Theta_{6}-\Theta_{4}}{\lambda(n)} \right).
 \end{align}

На підставі формули \eqref{Theorem1Asymp}  із Теореми~\ref{theorem1} та оцінок \eqref{Sum_psi(k)} та \eqref{form6}, отримуємо, що для всіх номерів $n$ таких, що виконується нерівність \eqref{alpha_1/4}  
\begin{align}\label{formTheor1}
{\cal E}_{n}(C^{\psi}_{\beta,1})_{C}=&
\frac{1}{\pi}\sum\limits_{k=n}^{\infty}\psi(k)
+
\Theta_{2}
\frac{1}{n}\sum\limits_{k=1}^{\infty}k\psi(k+n) \notag \\
=& 
 \frac{1}{\pi}\psi(n) \lambda(n) \left(1+\frac{\Theta_{4}}{\lambda(n)}+  \Theta_{5}\alpha(n)
 \right) 
 \!+\!\Theta_{2}\psi(n)  \lambda(n) \left(1+\frac{\Theta_{6}-\Theta_{4}}{\lambda(n)}+  (\Theta_{7}-\Theta_{5})\alpha(n)
 \right)
 \notag \\
=&
\psi(n)\lambda(n) \left( \frac{1}{\pi}+   \frac{\Theta_{4}/\pi +\Theta_{2}(\Theta_{6}-\Theta_{4})}{\lambda(n)} + \left(\frac{\Theta_{5}}{\pi}+ \Theta_{2}(\Theta_{7}-\Theta_{5})\ \right)\alpha(n) \right).
\end{align}

Оскільки для величини $\xi_{1}=\frac{\Theta_{4}}{\pi} +\Theta_{2}(\Theta_{6}-\Theta_{4})$ виконується оцінка
\begin{equation*}
-1\leq \xi_{1}\leq 1+\frac{1}{\pi},
 \end{equation*}
 а для  $\xi_{2}=\frac{\Theta_{5}}{\pi} +\Theta_{2}(\Theta_{7}-\Theta_{5})$ --- оцінка
\begin{equation*}
-4\leq \xi_{2}\leq \frac{4}{3}\left( 1+\frac{1}{\pi}\right),
 \end{equation*}
 то із \eqref{formTheor1} випливає \eqref{Theorem2Asymp}.
 Теорему~\ref{theorem2_M} доведено.
\end{proof}

Наведемо наслідок з Теореми~\ref{theorem2_M} у випадку, коли $\psi(t)=e^{-\alpha t^{-r}}$,  $\alpha>0$, $0<r\leq 1$, тобто коли класи $C^{\psi}_{\beta,1} $ є класами узагальнених інтегралів Пуассона 
$C^{\alpha,r}_{\beta,1}$. Легко переконатись, що для вказаних $\psi(t)$ при всіх $t\geq 1$,
\begin{equation}\label{lambda_corrolary_generPoisson}
\lambda(t)=\frac{t^{1-r}}{\alpha r}, \ \ \alpha(t)=\frac{1}{\alpha r t^{r}}.
\end{equation}

Із \eqref{lambda_corrolary_generPoisson} видно, що  умови  \eqref{alphaTo_0} i \eqref{lambdaTo_infty} Теореми~\ref{theorem2_M} виконуються. 
При цьому
виконання нерівності $\frac{1}{\alpha r n^{r}}\leq \frac{1}{4}$ рівносильне виконанню  умови \eqref{alpha_1/4}. Отже, з Теореми~\ref{theorem2_M} випливає наступне твердження.

\begin{corollary}\label{cor00}
Нехай $0<r<1$, $\alpha>0$, $\beta\in\mathbb{R}$, $n \in \mathbb{N}$. Тоді при всіх  $n\geq \left(\frac{4}{\alpha r} \right)^{\frac{1}{r}}$ справедлива рівномірно обмежена по  всіх розглядуваних параметрах оцінка
\begin{equation}\label{f112}
{\cal E}_{n}(C^{\alpha,r}_{\beta,1})_{C}= 
e^{-\alpha n^{r}}   n^{1-r} \Big(
\frac{1}{\pi\alpha r }+\mathcal{O}(1)\Big(\frac{1}{(\alpha r)^{2}}\frac{1}{n^{r}}+\frac{1}{n^{1-r}}\Big)\Big).
\end{equation}
\end{corollary}
Зазначимо, що оцінка вигляду \eqref{f112} при дещо жорсткіших обмеженнях на $n$ була знайдена у роботах \cite{SerdyukStepanyuk2016}-- \cite{SerdyukStepanyuk2018}. У зазаначених роботах містяться двосторонні оцінки величини $\mathcal{O}(1)$ через абсолютні сталі.

Наведемо ще декілька прикладів застосування Теореми~\ref{theorem2_M} для різних функцій $\psi$ із $\mathfrak{M}$, які задовольняють умовам \eqref{alphaTo_0} і \eqref{lambdaTo_infty}.

Будемо розглядати $\psi(t)$ вигляду
\begin{equation}\label{psi_1}
\psi(t)=(t+2)^{- \ln\ln (t+2)}, \ \ t\geq 1,
\end{equation}
\begin{equation}\label{psi_2}
\psi(t)=e^{- \ln^{2} (t+1)}, \ \ t\geq 1,
\end{equation}
\begin{equation}\label{psi_3}
\psi(t)=e^{- \frac{t+2}{\ln (t+2) }}, \ \ t\geq 1.
\end{equation}

Для зазначених функцій $\psi(t)$ знайдемо характеристики 
$\lambda(t)$  i $\alpha(t)$. Результати обчислень відображено в наступній таблиці:

\begin{center}\label{Tabl}
\begin{tabular}{|c|c|c|c|c|}
	\hline
№ & Функція $\psi(t)$ &  $ \alpha(t)$ & $\lambda(t)$    \\
\hline
1. & $(t+2)^{- \ln\ln (t+2)}$ & $\frac{t+2}{t}\frac{1}{1+\ln\ln(t+2)}$  & $\frac{t+2}{1+\ln\ln(t+2)}$   \\
&  &   &    \\
\hline
2. & $e^{- \ln^{2} (t+1)}$ & $\frac{t+1}{t}\frac{1}{2\ln(t+1)}$  & $\frac{t+1}{2\ln(t+1)}$   \\
&  &   &    \\
\hline
3. & $e^{- \frac{t+2}{\ln (t+2) }}$ &  $\frac{\ln^{2}(t+2)}{t(\ln(t+2)-1)}$ & $\frac{\ln^{2}(t+2)}{\ln(t+2)-1}$  \\
&  &   &    \\
\hline
\end{tabular}
\end{center}

Із Теореми~\ref{theorem2_M} і наведених в таблиці значень $\alpha(t)$ i $\lambda(t)$ отримуємо асимптотичні при $n\rightarrow\infty$ рівності для величин ${\cal E}_{n}(C^{\psi}_{\beta,1})_{C}$ у випадку, коли $\psi$ мають вигляд \eqref{psi_1}--\eqref{psi_3}.

\begin{corollary}\label{cor3}
Нехай $\psi(k)=(k+2)^{- \ln\ln (k+2)}$, $k=1,2,...,$ $\beta\in\mathbb{R}$ і $n\in\mathbb{N}$. Тоді при $n\rightarrow\infty$ виконується  асимптотична рівність
\begin{equation}\label{f13}
{\cal E}_{n}(C^{\psi}_{\beta,1})_{C}=
\frac{1}{\pi}\psi(n)\frac{n}{\ln\ln (n+2)}+\mathcal{O}(1)\psi(n) \frac{n}{(\ln \ln(n+2))^{2}}.
\end{equation}
\end{corollary}

\begin{corollary}\label{cor4}
Нехай  $\psi(k)=e^{-\ln ^{2}(k+1) }$,  $k=1,2,...,$ $\beta\in\mathbb{R}$ і $n\in\mathbb{N}$.Тоді при $n\rightarrow\infty$ має місце  асимптотична рівність
\begin{equation}\label{f16}
{\cal E}_{n}(C^{\psi}_{\beta,1})_{C}=
\frac{1}{2\pi}\frac{\psi(n) n}{\ln (n+1)}
+\mathcal{O}(1)\psi(n)\frac{n}{\ln^{2}(n+1)}.
\end{equation}
\end{corollary}

\begin{corollary}\label{cor5}
Нехай $\psi(k)=e^{- \frac{k+2}{\ln (k+2)}}$, $k=1,2,...,$ $\beta\in\mathbb{R}$ і $n\in\mathbb{N}$ Тоді при $n\rightarrow\infty$ має місце  асимптотична рівність
\begin{equation}\label{f20}
{\cal E}_{n}(C^{\psi}_{\beta,1})_{C}=
\frac{1}{\pi}\psi(n)\ln (n+2)+\mathcal{O}(1)\psi(n).
\end{equation}
\end{corollary}

Зауважимо, що у випадку, коли $\psi\in\mathfrak{M}$ і при $t\rightarrow\infty$ $\alpha(t)\rightarrow0$ i $\lambda(t)\rightarrow\infty$ за додаткової умови, що функція $\psi(t)$ є диференційовною скрізь на $[1, \infty)$, граничне співвідношення \eqref{LimitCase_BestPossibility}, яке гарантує той факт, що оцінки  \eqref{Theorem2Eq} і \eqref{Theorem1Asymp} є асимптотичними рівностями, завжди виконується.

Дійсно, застосувавши правило Лопіталя, маємо
\begin{equation}\label{Limit1}
\lim\limits_{n\rightarrow\infty}\frac{\int\limits_{n}^{\infty} \psi(t) dt}{\psi(n)} =
\lim\limits_{n\rightarrow\infty}\frac{ \psi(n) }{|\psi'(n)|} =
\lim\limits_{n\rightarrow\infty} \lambda(n) =\infty,
\end{equation}
\begin{equation}\label{Limit2}
\lim\limits_{n\rightarrow\infty}\frac{\int\limits_{n}^{\infty} t\psi(t) dt}{n\psi(n)} =
\lim\limits_{n\rightarrow\infty}\frac{ -n\psi(n) }{\psi(n) +n\psi'(n)} =
\lim\limits_{n\rightarrow\infty} \frac{\lambda(n)}{1-\alpha(n)} =\infty.
\end{equation}

Тоді, з урахуванням \eqref{Sum_psi_Int} i \eqref{Sum_k_psi_Formula1} мають місце асимптотичні рівності
\begin{equation}\label{Sum_psi_O_Estim}
\sum\limits_{k=n}^{\infty}\psi(k) = \int\limits_{n}^{\infty}\psi(t) dt + \mathcal{O}(1) \psi(n),
\end{equation}
\begin{equation}\label{Sum_k_psi_O_Estim}
\sum\limits_{k=n}^{\infty}k\psi(k) = \int\limits_{n}^{\infty}t\psi(t) dt + \mathcal{O}(1)n \psi(n).
\end{equation}

Використовуючи формули \eqref{Limit1}--\eqref{Sum_k_psi_Formula1} і застосувавши правило Лопіталя, отримуємо
\begin{align*}
&\lim\limits_{n\rightarrow\infty}\frac{\frac{1}{n}\sum\limits_{k=1}^{\infty}k\psi(k+n)}{\sum\limits_{k=n}^{\infty}\psi(k)} =
\lim\limits_{n\rightarrow\infty}\frac{\frac{1}{n}\sum\limits_{k=1}^{\infty}k\psi(k)-\sum\limits_{k=n}^{\infty}\psi(k)}{\sum\limits_{k=n}^{\infty}\psi(k)}
\notag \\
 =&
 \lim\limits_{n\rightarrow\infty}\frac{\frac{1}{n}\sum\limits_{k=1}^{\infty}k\psi(k)}{\sum\limits_{k=n}^{\infty}\psi(k)} -1
 =
 \lim\limits_{n\rightarrow\infty}\frac{\frac{1}{n}\int\limits_{n}^{\infty} t\psi(t) dt}{\int\limits_{n}^{\infty} \psi(t) dt }-1
 \end{align*}
 \begin{align}\label{Limit3}
 =&
 \lim\limits_{n\rightarrow\infty}\frac{-n\psi(n)}{\int\limits_{n}^{\infty} \psi(t) dt -n\psi(n)}-1
 =
 \lim\limits_{n\rightarrow\infty}\frac{-\int\limits_{n}^{\infty} \psi(t) dt }{\int\limits_{n}^{\infty} \psi(t) dt -n\psi(n)}
 \notag \\
 =&
 \lim\limits_{n\rightarrow\infty}\frac{\psi(n)}{-2\psi(n)-n\psi'(n)}
 =\lim\limits_{n\rightarrow\infty}\frac{  \frac{\psi(n)}{-n\psi'(n)} }{1-\frac{2\psi(n)}{-n\psi'(n)}}
 =\lim\limits_{n\rightarrow\infty}\frac{\alpha(n)}{1-\alpha(n)}=0.
\end{align}
Тим самим рівність \eqref{LimitCase_BestPossibility} доведено.

\section{Коментарі щодо нерівностей Лебега}\label{commentsLebesgue}

У підрозділах  \ref{corrolarySection_analyticFunctions} i \ref{Section_corrolary_infinitelyDifferentiable} були наведені наслідки з Теореми~\ref{theorem1} для швидко спадних послідовностей $\psi(k)$, для яких формула \eqref{Theorem1Asymp} є асимптотичною рівністю, або, що те саме, коли справджується \eqref{LimitCase_BestPossibility}. Зрозуміло, що у всіх розглянутих у підрозділах~\ref{corrolarySection_analyticFunctions} i \ref{Section_corrolary_infinitelyDifferentiable} частинних випадках для $\psi(\cdot)$ легко одержати і асимптотично непокращувані нерівності типу Лебега вигляду \eqref{Theorem2IneqCase1}. Ми обмежидись лише формулюванням лише деяких тверджень, які випливають із Теореми~\ref{theorem2}. Спочатку сформулюємо відповідні твердження для $\psi(t)=e^{-\alpha t^{r}}$, $\alpha>0$ i $r>0$.
Випадки $r>1$, $r=1$ i $r\in(0,1)$ виділяються окремо.

\begin{corollary}\label{corrolary_r>1}
Нехай $r>1$, $\alpha>1$ i $\beta\in\mathbb{R}$. Тоді при $n\geq \left(\frac{3}{\alpha r} \right)^{\frac{1}{r}}-1$ для довільної функції $f\in C^{\alpha,r}_{\beta}L_{1}$ має місце нерівність
\begin{equation}\label{corrolary_r>1_Ineq1}
\|f(\cdot)-S_{n-1}(f;\cdot)\|_{C}\leq
e^{-\alpha n^{r}} \left( \frac{1}{\pi}
+\mathcal{O}(1)e^{-\alpha n^{r-1}}\left(1+\frac{1}{\alpha r (n+1)^{r-2}} \right)
\right)
 E_{n}(f^{\alpha,r}_{\beta})_{L_{1}}.
 \end{equation}

 Крім того, для довільної функції $f\in  C^{\alpha,r}_{\beta}L_{1}$ можна знайти функцію  ${\mathcal{F}(x)=\mathcal{F}(f;n,x)}$ з множини $ C^{\alpha,r}_{\beta}L_{1}$ таку, що
  $E_{n}(\mathcal{F}^{^{\alpha,r}}_{\beta})_{L_{1}}=E_{n}(f^{\alpha,r}_{\beta})_{L_{1}}$ 
і  має місце  рівність
 \begin{equation}\label{corrolary_r>1_Eq}
\|\mathcal{F}(\cdot)-S_{n-1}(\mathcal{F}(\cdot);\cdot)\|_{C}= e^{-\alpha n^{r}} \left( \frac{1}{\pi}
+\mathcal{O}(1)e^{-\alpha n^{r-1}}\left(1+\frac{1}{\alpha r (n+1)^{r-2}} \right)
\right)
 E_{n}(f^{\alpha,r}_{\beta})_{L_{1}}.
 \end{equation}  
У \eqref{corrolary_r>1_Ineq1} i \eqref{corrolary_r>1_Eq}  $\mathcal{O}(1)$  --- рівномірно обмежена по всіх параметрах величина.
\end{corollary}

Аналоги нерівності \eqref{corrolary_r>1_Ineq1} і формули \eqref{corrolary_r>1_Eq}, яка доводить асимптотичну непокращуваність зазначеної нерівності, в яких залишкові члени записані в дещо іншій формі, отримано в \cite{MusienkoSerdyuk2013_5}.

\begin{corollary}\label{corrolary_r=1}
Нехай $\alpha>0$, $\beta\in\mathbb{R}$ i $n\in\mathbb{N}$. Тоді  для довільної функції $f\in C^{\alpha,1}_{\beta}L_{1}$ має місце нерівність
\begin{equation}\label{corrolary_r=1_Ineq1}
\|f(\cdot)-S_{n-1}(f;\cdot)\|_{C}\leq
\frac{1}{\pi}
\frac{e^{-\alpha n}}{1-e^{-\alpha}} 
 E_{n}(f^{\alpha,1}_{\beta})_{L_{1}}.
 \end{equation}

 Крім того, для довільної функції $f\in  C^{\alpha,1}_{\beta}L_{1}$ можна знайти функцію  ${\mathcal{F}(x)=\mathcal{F}(f;n,x)}$ з множини $ C^{\alpha,1}_{\beta}L_{1}$ таку, що
  $E_{n}(\mathcal{F}^{^{\alpha,1}}_{\beta})_{L_{1}}=E_{n}(f^{\alpha,1}_{\beta})_{L_{1}}$ 
і  має місце  рівність
 \begin{equation}\label{corrolary_r=1_Eq}
\|\mathcal{F}(\cdot)-S_{n-1}(\mathcal{F}(\cdot);\cdot)\|_{C}=e^{-\alpha n} 
\left( \frac{1}{\pi}
\frac{1}{1-e^{-\alpha}}+ \frac{\xi}{n} \frac{1}{(1-e^{-\alpha})^{2}} \right)
 E_{n}(f^{\alpha,1}_{\beta})_{L_{1}},
 \end{equation}  
 де для величини $\xi=\xi(f;n;\alpha;\beta)$ виконується нерівність 
$-2\leq \xi \leq 0$.
\end{corollary}

Оцінки \eqref{corrolary_r=1_Ineq1} i \eqref{corrolary_r=1_Eq} уточнюють оцінки рівномірних відхилень сум Фур'є на  множинах інтегралів Пуассона  $ C^{\alpha,1}_{\beta}L_{1}$, що були одержані в роботах \cite{SerdyukMusienko2010} i \cite{MusienkoSerdyuk2013_4}.

\begin{corollary}\label{corrolary_r<1}
Нехай $0<r<1$, $\alpha>0$, $\beta\in\mathbb{R}$ i $n\in\mathbb{N}$. Тоді при всіх $n\geq \left(\frac{4}{\alpha r} \right)^{\frac{1}{r}}$ для довільної функції $f\in C^{\alpha,r}_{\beta}L_{1}$ має місце нерівність
\begin{equation}\label{corrolary_r<1_Ineq1}
\|f(\cdot)-S_{n-1}(f;\cdot)\|_{C}\leq
e^{-\alpha n^{r}} n^{1-r} \left( \frac{1}{\pi \alpha r}
+\mathcal{O}(1)\left(\frac{1}{(\alpha r)^{2}}\frac{1}{n^{r}}+\frac{1}{n^{1-r}} \right)
\right)
 E_{n}(f^{\alpha,r}_{\beta})_{L_{1}}.
 \end{equation}

 Крім того, для довільної функції $f\in  C^{\alpha,r}_{\beta}L_{1}$ можна знайти функцію  ${\mathcal{F}(x)=\mathcal{F}(f;n,x)}$ з множини $ C^{\alpha,r}_{\beta}L_{1}$ таку, що
  $E_{n}(\mathcal{F}^{^{\alpha,r}}_{\beta})_{L_{1}}=E_{n}(f^{\alpha,r}_{\beta})_{L_{1}}$ 
і  має місце  рівність
 \begin{equation}\label{corrolary_r<1_Eq}
\|\mathcal{F}(\cdot)-S_{n-1}(\mathcal{F}(\cdot);\cdot)\|_{C}= e^{-\alpha n^{r}} n^{1-r}\left( \frac{1}{\pi \alpha r}
+\mathcal{O}(1)\left(\frac{1}{(\alpha r)^{2}}\frac{1}{n^{r}}+\frac{1}{n^{1-r}} \right)
\right)
 E_{n}(f^{\alpha,r}_{\beta})_{L_{1}}.
 \end{equation}  
У \eqref{corrolary_r<1_Ineq1} i \eqref{corrolary_r<1_Eq}  $\mathcal{O}(1)$  --- величини, що рівномірно обмежені по всіх параметрах.
\end{corollary}

При дещо жорсткіших обмеженнях на $n$ формули вигляду \eqref{corrolary_r<1_Ineq1} i \eqref{corrolary_r<1_Eq} були встановлені раніше в \cite{SerdyukStepanyukFilomat}.

\begin{corollary}\label{corrolary_Dq}
Нехай $\psi\in \mathcal{D}_{q}$, $q\in(0,1)$, $\beta\in\mathbb{R}$ i $n\in\mathbb{N}$. Тоді при всіх $n$ таких, що виконується нерівність \eqref{Varepselon_n_Dq} для довільної функції $f\in C^{\psi}_{\beta}L_{1}$ має місце нерівність
\begin{equation}\label{corrolary_Dq_Ineq1}
\|f(\cdot)-S_{n-1}(f;\cdot)\|_{C}\leq
\psi(n) \left( \frac{1}{\pi (1-q)}
+\mathcal{O}(1)\left(\frac{q}{n(1-q)^{2}}+\frac{\varepsilon_{n}}{(1-q)^{2}} \right)
\right)
 E_{n}(f^{\psi}_{\beta})_{L_{1}}.
 \end{equation}

 Крім того, для довільної функції $f\in  C^{\psi}_{\beta}L_{1}$ можна знайти функцію  ${\mathcal{F}(x)=\mathcal{F}(f;n,x)}$ з множини $ C^{\psi}_{\beta}L_{1}$ таку, що
  $E_{n}(\mathcal{F}^{^{\psi}}_{\beta})_{L_{1}}=E_{n}(f^{\psi}_{\beta})_{L_{1}}$ і таку, що при виконанні \eqref{Number_n_Dq} для неї має місце рівність
 \begin{equation}\label{corrolary_Dq_Eq}
\|\mathcal{F}(\cdot)-S_{n-1}(\mathcal{F}(\cdot);\cdot)\|_{C}= \psi(n) \left( \frac{1}{\pi (1-q)}
+\mathcal{O}(1)\left(\frac{q}{n(1-q)^{2}}+\frac{\varepsilon_{n}}{(1-q)^{2}} \right)
\right)
 E_{n}(f^{\psi}_{\beta})_{L_{1}}.
 \end{equation}  
У \eqref{corrolary_Dq_Ineq1} i \eqref{corrolary_Dq_Eq} величина $\varepsilon_{n}$ означена рівністю \eqref{Varepselon_n_Dq}, а  $\mathcal{O}(1)$  --- величини, що рівномірно обмежені по всіх параметрах.
\end{corollary}

\begin{theorem}\label{Theorem_Lebesgue_Inequality_M}
Нехай $\beta\in\mathbb{R}$, $\psi\in\mathfrak{M}$ i характеристики $\alpha(t)$ i $\lambda(t)$ вигляду \eqref{psi_alpha}
 i \eqref{psi_lambda} задовольняють умови \eqref{alphaTo_0} i \eqref{lambdaTo_infty}. Тоді для всіх  $n\in\mathbb{N}$ таких, що $\alpha(n)<\frac{1}{4}$ для будь-якої функції $f\in C^{\psi}_{\beta}L_{1}$ виконується нерівність 
\begin{equation}\label{Theorem_Lebesgue_Inequality_M_Ineq1}
\|f(\cdot)-S_{n-1}(f;\cdot)\|_{C}\leq
\psi(n) \lambda(n)\left( \frac{1}{\pi }
+\frac{\xi_{3}}{\lambda(n)}+\xi_{4}\alpha(n)
 \right)
 E_{n}(f^{\psi}_{\beta})_{L_{1}},
 \end{equation}
 де $0\leq \xi_{3}\leq \frac{4}{3\pi}$, $0\leq \xi_{4}\leq \frac{1}{\pi}$.

 Крім того, для довільної функції $f\in  C^{\psi}_{\beta}L_{1}$ можна знайти функцію  ${\mathcal{F}(x)=\mathcal{F}(f;n,x)}$ з множини $ C^{\psi}_{\beta}L_{1}$ таку, що
  $E_{n}(\mathcal{F}^{^{\psi}}_{\beta})_{L_{1}}=E_{n}(f^{\psi}_{\beta})_{L_{1}}$ і при $n\in\mathbb{N}$ таких, що $\alpha(n)<\frac{1}{4}$ 
   має місце рівність
 \begin{equation}\label{Theorem_Lebesgue_Inequality_M_Eq}
\|\mathcal{F}(\cdot)-S_{n-1}(\mathcal{F}(\cdot);\cdot)\|_{C}= \psi(n) \lambda(n)\left( \frac{1}{\pi }
+\frac{\xi_{5}}{\lambda(n)}+\xi_{6}\alpha(n)
 \right)
 E_{n}(f^{\psi}_{\beta})_{L_{1}},
 \end{equation}  
 де $-2\leq \xi_{5}\leq 2+ \frac{1}{\pi}$, $-8\leq \xi_{6}\leq \frac{4}{3}\left(2+ \frac{1}{\pi}\right)$.
\end{theorem}

Нерівність \eqref{Theorem_Lebesgue_Inequality_M_Ineq1} є наслідком формул \eqref{Theorem2Ineq1} та \eqref{Sum_psi(k)}, а рівність \eqref{Theorem_Lebesgue_Inequality_M_Eq} випливає із  \eqref{Theorem2Eq},  \eqref{Sum_psi(k)} та \eqref{form6}.

\end{document}